\documentclass[11pt,leqno]{amsart}
\usepackage{amssymb}
\usepackage{algorithm}
\usepackage{algorithmic}
\usepackage{amsmath}
\usepackage{appendix}
\usepackage{mathtools}
\usepackage{amsfonts}
\usepackage{color,xcolor}
\usepackage{graphicx}
\usepackage{manfnt}
\usepackage{enumerate}
\usepackage{comment}
\usepackage{epstopdf}
\usepackage{hyperref}
\newtheorem{definition}{Definition}[section]
\newtheorem{Thm}{Theorem}

\newtheorem{prop}[Thm]{Proposition}
\newtheorem{remark}{Remark}[section]
\newtheorem{lemma}[Thm]{Lemma}
\newtheorem{corollary}[Thm]{Corollary}

\newcommand{\tabincell}[2]{\begin{tabular}{@{}#1@{}}#2\end{tabular}}
 \numberwithin{equation}{section}
 \numberwithin{Thm}{section}

\newcommand{\be}{\begin{equation}}
\newcommand{\ee}{\end{equation}}
\newcommand\bes{\begin{eqnarray}}
\newcommand\ees{\end{eqnarray}}
\newcommand{\bess}{\begin{eqnarray*}}
\newcommand{\eess}{\end{eqnarray*}}

\title[Principal eigenvalue for a linear time-periodic parabolic operator]
{Asymptotics 
of the principal eigenvalue  for a linear time-periodic parabolic operator II: Small diffusion}
\author{Shuang Liu,\ \ Yuan Lou,\ \ Rui Peng\ \ and\ \, Maolin Zhou}

\thanks{{S. Liu}: Institute for Mathematical Sciences, Renmin University of China, Beijing 100872, China. Email: liushuangnqkg@ruc.edu.cn}

\thanks{{Y. Lou}: Department of Mathematics, Ohio State University, Columbus, OH 43210, USA. Email:  lou@math.ohio-state.edu}

\thanks{{R. Peng}: School of Mathematics and Statistics, Jiangsu Normal University,
Xuzhou,  Jiangsu 221116, China. Email:
pengrui\,$\b{}$\,seu@163.com}

\thanks{{M. Zhou}: Chern Institute of Mathematics and LPMC, Nankai University, Tianjin 300071, China. Email:
zhouml123@nankai.edu.cn}


\begin{document}
\maketitle

\begin{abstract}
We investigate the effect of  small diffusion on the principal eigenvalues
of linear time-periodic parabolic operators with  zero Neumann boundary conditions in
one dimensional space. The asymptotic behaviors
of the principal eigenvalues, as the diffusion coefficients tend to zero,
are established for non-degenerate and degenerate  spatial-temporally   varying environments.
 A new finding is the dependence of  these asymptotic behaviors  
on the periodic solutions of a specific ordinary differential equation induced by the drift.
The proofs are based upon delicate constructions of super/sub-solutions
and the applications of comparison principles.
\end{abstract}


\section{Introduction}\label{S1}
In this paper, we consider the following linear time-periodic parabolic eigenvalue problem  in one dimensional space:
\begin{equation}\label{SD_eq 1}
\begin{cases}
\smallskip
\partial_t\varphi-D\partial_{xx}\varphi-\partial_xm\partial_x\varphi+V\varphi=\lambda(D) \varphi &\text {in}\,\,(0,1)\times(0,T),\\
\smallskip
\partial_x\varphi(t,0)=\partial_x\varphi(1,t)=0 &\text {on}\,\,[0,T],\\
\varphi(x,0)=\varphi(x,T)  &\text{on}\,\,(0,1),
\end{cases}
\end{equation}
where  $D>0$ 
represents the diffusion rate, and the functions $m\in C^{2,1}([0,1]\times[0,T])$ and $V\in C([0,1]\times[0,T])$ are assumed to be  periodic in $t$ with  a common period $T$.

By the Krein-Rutman Theorem,   \eqref{SD_eq 1} admits a simple and real eigenvalue (called  \emph{principal  eigenvalue}), denoted by $\lambda(D)$, which corresponds to a positive eigenfunction (called
\emph{principal  eigenfunction}) and satisfies $\mathrm{Re}\lambda>\lambda(D)$ for any other eigenvalue $\lambda$ of  \eqref{SD_eq 1}; see Proposition 7.2 of \cite{Hess}.
The principal eigenvalue $\lambda(D)$ plays  a fundamental role in the study of reaction-diffusion equations  and systems in spatio-temporal  media, e.g. in the stability analysis for equilibria \cite{CC2003, CN2019, Hess,Hutson2001}.
Of particular interest is to understand
the dependence of $\lambda(D)$ on the parameters \cite{ LamLou2016,LL2019, Nadin2009, Nadin2011}.
The present paper continues our previous studies in \cite{LLPZ20191,LLPZ20192}  on the principal eigenvalues
for time-periodic parabolic operators, where  the dependence of $\lambda(D)$ on frequency and  advection rate were investigated. Our main goal here
is to establish the asymptotic behavior of $\lambda(D)$ as the diffusion rate $D$ tends to zero.



For notational convenience, given any  $T$-periodic function  $p(x,t)$,  we define 
$$\hat{p}(x):=\frac{1}{T}\int_{0}^{T}p(x,s)\,\mathrm{d}s \quad \text{ and }\quad p_+(x,t):=\max\left\{p(x,t),0\right\},$$
and redefine $\partial_{xx}\hat{m}(0)$ and $\partial_{xx}\hat{m}(1)$  via
 \begin{equation}\label{boundarycondition}
  \partial_{xx}\hat{m}(0)=\begin{cases}
-\infty & \text{if }\partial_x\hat{m}(0^+)<0,\\
\infty & \text{if }\partial_x\hat{m}(0^+)>0,
\end{cases}\quad\text{and}\quad
\partial_{xx}\hat{m}(1)=\begin{cases}
\infty & \text{if }\partial_x\hat{m}(1^-)<0,\\
-\infty & \text{if }\partial_x\hat{m}(1^-)>0.
\end{cases}
 \end{equation}


For the case when $V$ and $\partial_xm$ depend upon the space variable alone, i.e.
$V(x,t)=V(x)$ and $\partial_xm(x,t)=m'(x)$,
problem \eqref{SD_eq 1} reduces to
the following elliptic eigenvalue problem:
\begin{equation}\label{elliptic}
 \begin{cases}
 \smallskip
-D\varphi''-  m'(x)\varphi'+V(x)\varphi=\lambda(D)\varphi\ \ &\text{in}\,\,(0,1),\\
\varphi'(0)=\varphi'(1)=0.
\end{cases}
 \end{equation}
 This sort of advection-diffusion operator in \eqref{elliptic}
 with small diffusion can be regarded as a singular perturbation of the corresponding first order operator \cite{W1975}, and was 
 studied  in \cite{FW1984} by the large deviation approach. Therein, the limit of the principal eigenvalue $\lambda(D)$ as $D\to0$ 
 plays a pivotal  role in studying the large time behavior of the trajectories of stochastic systems; see also \cite{DEF1974, F1973}. 
 Recently the asymptotic behavior of $\lambda(D)$ for problem \eqref{elliptic} has been
 considered in 
 \cite{CL2012} for general bounded domains, and their result in  particular implies
 \begin{Thm}\cite{CL2012}\label{independentt}
 Assume $V(x,t)=V(x)$ and $\partial_xm(x,t)=m'(x)$.
 Suppose that $ m'(0)\neq0$, $ m'(1)\neq0$, and all critical points of $m$ are non-degenerate. 
 Then
$$\lim_{D\rightarrow 0}\lambda(D)=
\min_{x\in\Sigma\cup\{0,1\}}\left\{V\left(x\right)+[m'']_+\left(x\right) \right\},
$$
where $\Sigma:=\{x\in(0,1):m'(x)=0\}$ and $m''(0),m''(1)$ are defined by \eqref{boundarycondition}.
 \end{Thm}
 We refer to \cite{PZZ2018} for recent progress on problem \eqref{elliptic} under general boundary conditions.

 Theorem \ref{independentt} indicates that the limit of $\lambda(D)$ relies  upon   
 the set of critical points of function $m$
 in the elliptic scenario. Turning to the time-periodic parabolic case  where $m$  depends on
 both spatial and temporal variables,  it seems reasonable to anticipate that  the limit of $\lambda(D)$ will be associated to the curves $x(t)$ satisfying $\partial_xm(x(t),t)=0$. This is indeed the case for  the limit of the principal eigenvalue
 with large advection, and we refer to Theorem 1.1 in \cite{LLPZ20192} for further details.
 However, it turns out that this is generally not true while considering the limit
 of $\lambda(D)$ as $D$ tends to zero. Instead, the asymptotic behavior of $\lambda(D)$
 depends heavily on the periodic solutions of the following  ordinary differential equation:
  \begin{equation}\label{eq:P}
  \begin{cases}
  \smallskip
\dot{P}(t)=-\partial_xm\left(P(t),t\right),\\
P(t)=P(t+T).
 \end{cases}
 \end{equation}
 More specifically, our main result can be stated as follows.

\begin{Thm}\label{sdthmP}
Assume that $\partial_xm(0,t)\neq0$  and $\partial_xm(1,t)\neq0$  for all $t\in[0,T]$. Let $\partial_{xx}\hat{m}(0)$ and $\partial_{xx}\hat{m}(1)$ be defined by \eqref{boundarycondition}.\smallskip
\\
  \noindent{\rm(i)} If \eqref{eq:P} has at least one but finite many 
  $T$-periodic solutions,
  denoted by $\{P_i(t)\}_{i=1}^{N}$,  satisfying  
$0<P_1(t)<\ldots<P_N(t)<1$, and $\partial_{xx}m\left(P_i(t),t\right)\neq 0$ for $1\leq i\leq N$ and $t\in[0,T]$, 
then
$$\lim_{D\rightarrow 0}\lambda(D)=
\min_{0\leq i\leq N+1}\left\{\frac{1}{T}\int_0^T \Big[V\left(P_i(s),s\right)+[\partial_{xx}m]_+\left(P_i(s),s\right)\Big]\mathrm{d}s \right\}, 
$$
 where $P_0(t)\equiv 0$ and $P_{N+1}(t)\equiv 1$;

 \noindent{\rm(ii)} If \eqref{eq:P} has no periodic solutions,
 then
  $$\lim_{D\rightarrow 0}\lambda(D)=
\min\Big\{ \hat V(0)+[\partial_{xx}\hat{m}]_+\left(0\right),\,\,   \hat V(1)+[\partial_{xx}\hat{m}]_+\left(1\right) \Big\}.
$$
\end{Thm}
If $V$ and $m$ are independent of time,  all solutions of \eqref{eq:P} are constants which correspond to the critical points of function $m$, and part {\rm(i)} of Theorem \ref{sdthmP} is reduced to Theorem \ref{independentt}.
When $m(x,t)$ is monotone in $x$, part {\rm(ii)} of Theorem \ref{sdthmP} was first established in \cite{PZ2015}.

One potential application of Theorem \ref{sdthmP} is the  study of large-time behaviours of solutions to the Cauchy problem for singularly perturbed parabolic equations in  spatio-temporal  media \cite{BH1984,FM2018,Hess}, in which the growth or decay rate of the solutions can be described in terms of  $\lambda(D)$.
In a very recent work \cite{FM2019}, the asymptotics of $\lambda(D)$ for small $D$  was considered in a  case of underlying advection $\partial_x m$ being a constant, when analyzing the effect of small mutations on phenotypically-structured populations in a shifting and fluctuating environment.

The restriction $\partial_{xx}m\left(P(t),t\right)\neq 0$
in Theorem \ref{sdthmP}, in fact guarantees the non-degeneracy of advection $\partial_x m$ along  periodic solution $P$ of \eqref{eq:P}. See \cite{CL2008,LLPZ20192} for the definitions of degeneracy and non-degeneracy. To complement Theorem \ref{sdthmP},  we consider
a type of degenerate advection 
in the following result:
\begin{Thm}\label{thm1.1}
 Suppose that for each $1\leq i\leq N$, $\partial_x m(\kappa_i, t)\equiv0$ for all $t\in[0,T]$, and
 $0<\kappa_1<\cdots<\kappa_{N}<1$. Furthermore, assume that $\{i:0\leq i\leq N\}=\mathbf{A}\cup\mathbf{B}$, where
 \begin{itemize}
  \item [] $\mathbf{A}=\big\{i:\ \ 0\leq i\leq N,\,\, \partial_xm(x,t)\neq0,\ \,
      (x,t)\in(\kappa_{i},\kappa_{i+1})\times[0,T]\big\};$
  \item [] $\mathbf{B}=\big\{i:\ \ 0\leq i\leq N,\,\,\partial_x m(x,t)\equiv0,\ \,
      (x,t)\in[\kappa_{i},\kappa_{i+1}]\times[0,T]\big\},$
\end{itemize}
 where $\kappa_0=0$ and $\kappa_{N+1}=1$.
Then we have
\begin{equation}\label{limitlambdaD}
  \lim_{D\rightarrow 0}\lambda(D)=\min\left\{\min_{0\leq i\leq N+1}\left\{\hat{V}(\kappa_i)+[\partial_{xx}\hat{m}]_+(\kappa_i)\right\},\,\min_{i\in\mathbf{B}}\left\{\min_{x\in[\kappa_i,\kappa_{i+1}]}\hat{V}(x)\right\}\right\},
\end{equation}
where  $\partial_{xx}\hat{m}(0)$ and $\partial_{xx}\hat{m}(1)$ are defined by \eqref{boundarycondition}.
\end{Thm}

The main contribution of Theorem \ref{thm1.1} is to allow  $\mathbf{B}\neq\emptyset$, i.e. the spatial-temporal  degeneracy of function $\partial_x m$.
When $\mathbf{B}=\emptyset$, which means $\partial_xm(x,t)\neq0$ for all $x\neq \kappa_i, 0\leq i\leq N+1$,  all solutions of \eqref{eq:P} are nothing but constant solutions $P\equiv\kappa_i, 1\leq i\leq N$, and consequently, Theorem \ref{thm1.1} becomes a special case of Theorem \ref{sdthmP} when $\mathbf{B}=\emptyset$.

The assumption $i\in \mathbf{A}$  implies there are no periodic solutions of \eqref{eq:P} in $[\kappa_i,\kappa_{i+1}]\times[0,T]$ except for constant solutions $P\equiv\kappa_i$ and $P\equiv\kappa_{i+1}$. Without this assumption, the situation becomes
even more complicated. 
To illustrate the complexity, we consider the special case
$m(x,t)=\alpha b(t)x$ as in \cite{LLPZ20192}, 
where $\alpha>0$ denotes the advection rate, and the $T$-periodic function $b$ is  Lipschitz continuous. 
In this case, problem \eqref{SD_eq 1} becomes
\begin{equation}\label{SD_eq 4}
\begin{cases}
\partial_t\varphi-D\partial_{xx}\varphi- \alpha b(t)\partial_x\varphi+V\varphi=\lambda(D) \varphi & \text{in } (0,1)\times[0,T],\\
\partial_x\varphi(0,t)=\partial_x\varphi(1,t)=0 &\text{on }[0,T],\\
\varphi(x,0)=\varphi(x,T) &\text{on } (0,1).
\end{cases}
\end{equation}

For different $\alpha$ and $b$,  we have the following result:

\begin{Thm}\label{sdthm4}
Let $\lambda(D)$ denote the principal eigenvalue of \eqref{SD_eq 4}. \smallskip \\
   \noindent{\rm(i)} If $\hat{b}\neq0$, then for all $\alpha>0$,
  $$\lim\limits_{D\rightarrow0}\lambda(D)=\begin{cases}
  \hat{V}(1) &\text{for }\,\,\hat{b}>0,\\
  \hat{V}(0) &\text{for }\,\,\hat{b}<0;
  \end{cases}$$

 \noindent{\rm(ii)} If $\hat{b}=0$, set $P(t)=-\int_{0}^{t}b(s)\mathrm{d}s$,
  $\overline{P}=\max_{[0,T]}P$, and $\underline{P}=\min_{[0,T]}P.$
      Then
      $$\lim\limits_{D\rightarrow0}\lambda(D)=\begin{cases}
      \medskip
      \min\limits_{y\in[-\alpha\underline{P}, \, 1-\alpha\overline{P}]}\left\{\frac{1}{T}\int_{0}^{T}V(\alpha P(s)+y,s)\mathrm{d}s\right\}, &0<\alpha\leq \frac{1}{\overline{P}-\underline{P}},\\
      \frac{1}{T}\int_{0}^{T}V( \tilde{P}_\alpha(s),s)\mathrm{d}s, &\alpha>\frac{1}{\overline{P}-\underline{P}},
      \end{cases}
$$
       where $\tilde{P}_\alpha\in C([0,T];[0,1])$ is the unique $T$-periodic solution  of $\dot{\tilde{P}}(t)=-\alpha F(\tilde{P}(t),t)$ in $[0,1]$, and $F$ 
       is given  by
 \begin{equation}\label{definition_F}
 F(x,t)=\begin{cases}
  b(t)  &0<x<1, \, t\in [0, T],
  \\
  \min\{b(t), 0\},  &x=0, \, t\in [0, T],
  \\
  \max\{b(t), 0\},   &x=1, \, t\in [0, T].
  \end{cases}
  \end{equation}
\end{Thm}

 \begin{remark}
{\rm   When $\hat{b}=0$ and $\alpha=\frac{1}{\overline{P}-\underline{P}}$, part {\rm(ii)} of Theorem {\rm \ref{sdthm4}} implies  that $\lambda(D)\to\frac{1}{T}\int_{0}^{T}V\left(\tfrac{P(s)-\underline{P}}{\overline{P}-\underline{P}},s\right)\mathrm{d}s$ as $D\rightarrow0$.
Direct calculation yields that $\tfrac{P(t)-\underline{P}}{\overline{P}-\underline{P}}$ is in fact a periodic solution of $\dot{\tilde{P}}(t)=- \frac{1}{\overline{P}-\underline{P}} F(\tilde{P}(t),t)$, so that the uniqueness  part in Lemma {\rm \ref{existenceandunique}} implies  $\tilde{P}_\alpha(t)\to\tfrac{P(t)-\underline{P}}{\overline{P}-\underline{P}}$ as $\alpha\to \frac{1}{\overline{P}-\underline{P}}$.
This means that the limit of $\lambda(D)$ as $D\to 0$ is continuous  at $\alpha=\frac{1}{\overline{P}-\underline{P}}$.
}
\end{remark}

 For  $m(x,t)=\alpha b(t) x$, Theorem \ref{sdthm4}  gives a complete description of the behaviors of $\lambda(D)$ as $D\to0$, and it provides a type of complicated spatial-temporal degeneracy not covered by Theorem \ref{thm1.1}. To further illustrate Theorem \ref{sdthm4},
 consider the case $b(t)=-\frac{\pi}{T}\sin\left(\frac{2\pi t}{T}\right)$, in which
 $$P(t)=\tfrac{1}{2}\cos\left(\tfrac{2\pi t}{T}\right)-\tfrac{1}{2}, \quad\overline{P}=0, \quad \underline{P}=-1.$$
 More precisely, {\rm{(i)}} when $0<\alpha< 1$, we could find some $y_\alpha\in[\alpha,1]$ such that $\lambda(D)\rightarrow\frac{1}{T}\int_{0}^{T}V(\alpha P(s)+y_\alpha,s)\mathrm{d}s$ as $ D\rightarrow0$, and the trajectory $\{\alpha P(t)+y_\alpha:t\in[0,T]\}$ in $x$-$t$ plane is illustrated by the red solid curve in  Fig.\ref{figure1.1}{\rm(a)}, where the two red dotted curves represent $\{\alpha P(t)+\alpha:t\in[0,T]\}$ and $\{\alpha P(t)+1:t\in[0,T]\}$, respectively; {\rm{(ii)}} When $\alpha=1$,  we have $\lambda(D)\rightarrow\frac{1}{T}\int_{0}^{T}V( P(s)+1,s)\mathrm{d}s$ as $ D\rightarrow0$, and the trajectory $\{P(t)+1:t\in[0,T]\}$ is shown  in Fig.\ref{figure1.1}{\rm(b)}; {\rm{(iii)}} When $\alpha>1$, it follows that $\lambda(D)\rightarrow\frac{1}{T}\int_{0}^{T}V( \tilde{P}_\alpha(s),s)\mathrm{d}s$, and the corresponding trajectory $\{\tilde{P}_\alpha(t):t\in[0,T]\}$ is given in Fig.\ref{figure1.1}{\rm(c)}-{\rm(d)}.
\begin{figure}[http!!]
  \centering
\includegraphics[height=2.0in]{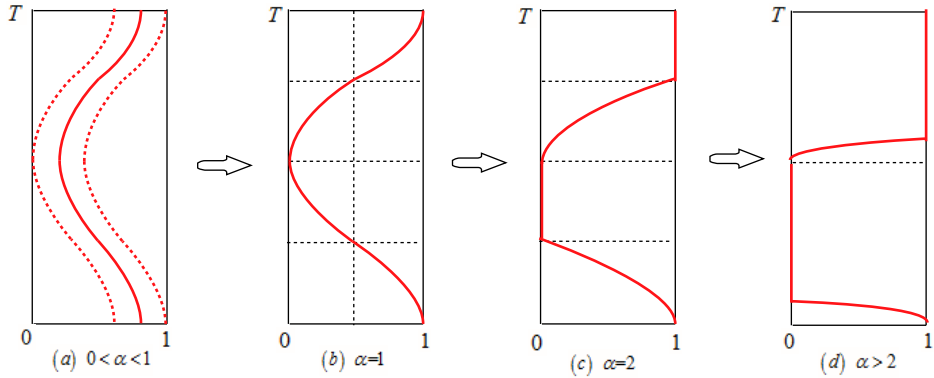}
  \caption{Each rectangle corresponds to the region $[0,1]\times[0,T]$ in  $x$-$t$ plane.
  The limit of $\lambda(D)$ as $D\to 0$ is determined by the  average of $V$ over the red solid curves, illustrated   for various ranges of $\alpha$ and  $m(x,t)=-\frac{\alpha\pi x}{T}\sin\left(\frac{2\pi t}{T}\right)$.}\label{figure1.1}
  \end{figure}



As the proofs of Theorems \ref{sdthmP}, \ref{thm1.1}, and \ref{sdthm4} are  fairly technical, in the following we briefly outline the main strategies in proving Theorems \ref{sdthmP} and \ref{thm1.1}:
\begin{itemize}

\smallskip
\item[{\rm{(i)}}]  We note that 
 $\lambda(D)$ for \eqref{elliptic} in the elliptic situation can be characterized by
 variational formulation 
\cite{CL2008,CL2012, PZZ2018,PZ2018}. In contrast, the time-periodic parabolic problem \eqref{SD_eq 1} has no variational formulations.
Our general strategy is to 
construct super/sub-solutions and apply generalized comparison principle developed in \cite[Theorem A.1]{LLPZ20192}.
This technique was first introduced by Berestycki and  Lions \cite{BL1980}
to  the elliptic scenario,  whereas its adaptation to our context is more subtle 
because of the presence of temporal variable; see \cite{PZ2015} for further discussions.

\smallskip
\item [{\rm{(ii)}}] We first  establish Theorem \ref{thm1.1}  which assumes that 
$\partial_x m$ is strictly positive,  negative, or identically zero  in each
sub-interval $(\kappa_i, \kappa_{i+1})$. 
The main difficulty is to establish the lower bound of the principal eigenvalue 
in \eqref{limitlambdaD}.
The construction of super-solutions
near the curves  $\{(\kappa_i,t),t\in[0,T]\}$ 
is rather subtle, due to the fact that the spatio-temporal derivatives of the principal eigenfunction of  \eqref{SD_eq 1} restricted to  the curves may be unbounded as $D$ tends to zero. Our strategy is to construct  the super-solution  almost coinciding with the principal eigenfunction of \eqref{SD_eq 1}  near these curves,
and then use  an  iterated argument to extend the super-solution to the whole domain. 

\smallskip
\item [{\rm{(iii)}}] A key ingredient in the proof of Theorem \ref{sdthmP}
is to recognize the critical role of the solutions of \eqref{eq:P}. 
Our idea is to reduce the proof of Theorem \ref{sdthmP} to that of 
Theorem \ref{thm1.1} with $\mathbf{B}=\emptyset$.
As Theorem \ref{thm1.1} assumes that $\partial_x m$ is either strictly  positive or  negative in each
sub-interval $(\kappa_i, \kappa_{i+1})$, there are two difficulty in doing so: First,
the solutions $P_i(t)$ of \eqref{eq:P} are not constant ones as specified  in Theorem \ref{thm1.1}.
This difficulty can be overcome by introducing a proper
transformation so that 
$P_i(t)$ become constant after the transformation. 
The second difficulty is that
a {\it priori} we do not know the sign of the term $\partial_x m$ 
in each
 $(\kappa_i, \kappa_{i+1})$.
  Our idea is to introduce another transformation,  which is associated with the trajectories of \eqref{eq:P}.
We prove that after the second transformation,
$\partial_x m$  is indeed either  strictly positive or negative in each
$(\kappa_i, \kappa_{i+1})$,
so that the proof of 
Theorem \ref{thm1.1} is directly applicable to complete the proof of Theorem \ref{sdthmP}.
\end{itemize}

\smallskip

This paper is organized as follows: In Section \ref{S2} we present some results associated with the case when all of periodic solutions of \eqref{eq:P} are constants and establish Theorem \ref{thm1.1}. These results are used in Section \ref{S3} to give the proof of Theorem \ref{sdthmP}, by combining with an idea of ``straightening periodic solutions''. Section \ref{S4} is devoted to the proof of Theorem \ref{sdthm4}. A generalized comparison result will be
presented 
in the Appendix.

\medskip

\section{Proof of Theorem \ref{thm1.1}}\label{S2}
This section is devoted to the proof of Theorem \ref{thm1.1}. 
Hereafter, we use
$\mathcal{L}_{D}$ to denote the time-periodic parabolic operator
$$\mathcal{L}_{D}:=\partial_{t}-D\partial_{xx}- \partial_xm\partial_x+V.$$

For any $x\in[0,1]$, we define a $T$-periodic function $f_x:[0,T]\to(0,\infty)$ by
\begin{equation}\label{deff}
  f_x(t)=\mbox{exp} \left[-\int_0^tV(x,s)\mathrm{d}s+\hat{V}(x)t\right],
\end{equation}
which solves, for fixed $x\in[0,1]$, that $(\log f_x)'=\hat{V}(x)-V(x,t)$.
\begin{prop}\label{sdlem1}
For any constant $\kappa\in(0,1)$, suppose that
$$\begin{cases}
\partial_xm(x,t)>0,& (x,t)\in[0,\kappa)\times[0,T],\\
\partial_xm(x,t)<0,& (x,t)\in(\kappa,1]\times[0,T].\\
\end{cases}$$
Then we have
$$\lim_{D\rightarrow 0}\lambda (D)=\hat{V}(\kappa).$$
\end{prop}
\begin{proof}

We first prove the upper bound
\begin{equation}\label{ldn1}
 \limsup_{D\to0}\lambda(D)\leq\hat{V}(\kappa).
\end{equation}

Fix any $\epsilon>0$. For sufficiently small $D$, we  construct a strict non-negative sub-solution $\underline \varphi$ in the sense of Definition \ref{appendixldef} (see Appendix A) such that
 \begin{equation}\label{ldn2}
 \begin{cases}
 \smallskip
 \mathcal{L}_D\underline{\varphi}\leq\left[\hat{V}(\kappa)+\epsilon\right]\underline{\varphi}
 &\text {in}\,\,((0,1)\setminus\mathbb{X})\times(0,T),\\
 \medskip
 \partial_x\underline{\varphi}(0,t)=
 \partial_x\underline{\varphi}(1,t)=0  &\text {on}\,\,[0,T], \\
 \underline{\varphi}(x,0)=\underline{\varphi}(x,T) &\text{on}\,\,(0,1),
 \end{cases}
\end{equation}
 for some point set $\mathbb{X}$ determined later.

To this end, by continuity of $V$,  we  choose  small $\delta\in(0,1)$ such that
\begin{equation}\label{liu0004}
    |V(x,t)-V(\kappa,t)|<\epsilon/2\quad\text{on} \quad [\kappa-\delta,\kappa+\delta]\times[0,T].
\end{equation}
 Then we define  $\underline{\varphi}$ 
 by
 $$\underline{\varphi}(x,t):=f_{\kappa}(t) \cdot \underline{z}(x),$$
 where $f_\kappa(t)$ is defined in \eqref{deff} with $x=\kappa$, %
 and $\underline{z}\in C([0,1])$ is given by
\begin{equation}\label{underlinez}
  \underline{z}(x):=\begin{cases}
  \smallskip
-(x-\kappa)^2+\delta ^2 &\text{ on } [\kappa-\delta,\kappa+\delta],\\
0 & \text{ on } [0,\kappa-\delta)\cup(\kappa+\delta,1].
\end{cases}
\end{equation}

Observe that
$\partial_x\underline{\varphi}\left(\left(\kappa\pm\delta\right)^+,\cdot\right)>\partial_x\underline{\varphi}\left(\left(\kappa\pm\delta\right)^-,\cdot\right)$.  We now identify $\mathbb{X}$ in \eqref{ldn2} as
$$\begin{array}{c}
\mathbb{X}=\{\kappa\pm\delta\}.
\end{array}$$
To verify \eqref{ldn2}, 
 note from definition \eqref{deff} that $f'_{\kappa}=(\hat{V}(\kappa)-V(\kappa,t))f_{\kappa}$,
direct calculations on $[\kappa-\delta,\kappa+\delta]\times[0, T]$ yield that for small $D$,
\begin{equation*}
\begin{split}
\mathcal{L}_D\underline{\varphi}
&=  f'_{\kappa}(t)  \underline{z}-D \partial_{xx} \underline{\varphi}- \partial_xm\partial_x\underline{\varphi} +V(x,t) \underline{\varphi} \\
&= \left[\hat{V}(\kappa)-V\left(\kappa,t\right)+V(x,t)\right]\underline{\varphi}-\partial_{x}m\partial_{x}\underline{\varphi}-D\partial_{xx}\underline{\varphi}\\
&\leq \left[\hat{V}(\kappa)+\epsilon/2\right]\underline{\varphi}-\partial_{x}m\partial_{x}\underline{\varphi}-D\partial_{xx}\underline{\varphi}\\
&\leq \left[\hat{V}(\kappa)+\epsilon\right]\underline{\varphi},
\end{split}
\end{equation*}
 where $-\partial_{x}m\partial_{x}\underline{\varphi}-D\partial_{xx}\underline{\varphi}\leq \frac{\epsilon}{2} \underline{\varphi}$ in the last inequality is due to the fact that $-\partial_{x}m\partial_{x}\underline{\varphi}<0\leq \epsilon\underline{\varphi}$ in the neighborhoods of $\{\kappa\pm\delta\}\times[0,T]$. Hence \eqref{ldn2} holds, and  \eqref{ldn1} follows from \eqref{ldn2} and Proposition \ref{appendixprop} by letting $\epsilon\to0^+$.

Next, we show that
\begin{equation}\label{ldn1.1}
 \liminf_{D\to0}\lambda(D)
 \geq\hat{V}(\kappa).
\end{equation}

Define $\overline{\varphi}\in C^{2,1}([0,1]\times[0,T])$ by
 $$\overline{\varphi}(x,t):=f_\kappa(t)\cdot e^{M_1(x-\kappa)^2}$$
 with  $M_1>0$ to be specified later. For any given $\epsilon>0$, we  shall choose  $M_1$ large so that for sufficiently small $D$, $\overline{\varphi}$ satisfies
 \begin{equation}\label{supersoltion100}
 \begin{cases}
  \smallskip
 \mathcal{L}_D\overline{\varphi}\geq \left[\hat{V}(\kappa)-\epsilon\right]\overline{\varphi} &\text {in}\,\,(0,1)\times(0,T),\\
 \smallskip
 \partial_x\overline{\varphi}(0,t)<0 <\partial_x\overline{\varphi}(1,t)  &\text {on}\,\,[0,T], \\
 \overline{\varphi}(x,0)=\overline{\varphi}(x,T) &\text{on}\,\,(0,1).
 \end{cases}
\end{equation}

To establish \eqref{supersoltion100}, we first recall that $\delta$ is chosen as in \eqref{liu0004}. For $x\in(0,\kappa-\delta]\cup[\kappa+\delta,1)$, there exists some $\epsilon_0>0$  such that $|\partial_xm|\geq\epsilon_0$,
and thus
\begin{equation*}
  -\partial_xm\partial_x(\log\overline{\varphi})=2M_1\partial_xm\cdot(x-\kappa)\geq 2M_1\delta\epsilon_0,
\end{equation*}
from which direct calculation leads to
\begin{equation}\label{liu0002}
\begin{split}
\mathcal{L}_D\overline{\varphi}&\geq\left\{\hat{V}(\kappa)+V(x,t)-V\left(\kappa,t\right)-D\left[2M_1+4M_1^2(x-\kappa)^2\right]+2M_1\delta\epsilon_0\right\}\overline{\varphi}. 
\end{split}
\end{equation}
We  choose    $M_1$ large such that $2M_1\delta\epsilon_0>2\|V\|_{L^\infty}$. Letting  $D$ be small enough in \eqref{liu0002}, we deduce $
\mathcal{L}_D\overline{\varphi}\geq \hat{V}(\kappa)\overline{\varphi}$ as desired.

 For $x\in[\kappa-\delta,\kappa+\delta]$,  by $-\partial_xm\partial_x\overline{\varphi}\geq0$ and the definition of $\delta$ we have
\begin{equation*}
\begin{split}
\mathcal{L}_D\overline{\varphi}\geq \left\{\hat{V}(\kappa)+V(x,t)-V\left(\kappa,t\right)-D\left[2M_1+4M_1^2(x-\kappa)^2\right]\right\}\overline{\varphi}\geq \left[\hat{V}(\kappa)-\epsilon\right]\overline{\varphi}
\end{split}
\end{equation*}
for sufficiently small $D$.

 Therefore, \eqref{supersoltion100} holds and \eqref{ldn1.1} follows from Proposition \ref{appendixprop} with $\mathbb{X}=\emptyset$.
\end{proof}

%
To proceed further, we will need the following result:
\begin{lemma}\label{wholespaceeigenvalue0}
Let $\rho(t)\geq0\,(\not\equiv0)$ be any $T$-periodic function. For each  $R>0$, denote by $\mu_R$  the principal eigenvalue of the following problem:
\begin{equation}\label{unboundeigenvalue}
\begin{cases}
\smallskip
\partial_t\varphi-\partial_{xx}\varphi- x\rho(t)\partial_x\varphi=\mu_R \varphi &\text {in}\,\,(-R,R)\times(0,T),\\
\smallskip
\varphi(-R,t)=\varphi(R,t)=0 &\text {on}\,\,[0,T], \\
 \varphi(x,0)=\varphi(x,T) &\text{on}\,\,[-R,R].
\end{cases}
\end{equation}
Then we have
$$\lim_{R\to\infty}\mu_R=\hat{\rho}. 
$$
\end{lemma}
\begin{proof}
 For each $\delta\geq 0$, in view of $\rho(t)\geq0\,(\not\equiv0)$ in $[0,T]$,  we choose  $\beta_\delta(t)$ as the unique positive 
solution of the problem
  \begin{equation}\label{beta}
  \begin{cases}
  \smallskip
 \frac{\dot{\beta}(t)}{2}=\beta(t)\left[\rho(t)+\delta-\beta(t)\right]\,\,\text{ in }\,[0,T],\\
 \beta(0)=\beta(T).
 \end{cases}
 \end{equation}
Denote by $(\alpha_\delta(t),\mu_\delta)$
an eigenpair, with $\alpha_\delta(t)>0$, of the eigenvalue problem
 \begin{equation}\label{alpha}
  \begin{cases}
  \smallskip
 \dot{\alpha}(t)+\beta_\delta(t)\alpha(t)=\mu \alpha(t) \,\,\text{ in }[0,T],\\
 \alpha(0)=\alpha(T).
 \end{cases}
 \end{equation}
  Dividing both sides of \eqref{beta} by $\beta_\delta$, and integrating the resulting equation over $[0,T]$,
by periodicity of $\beta_\delta$ we have
 $\hat{\beta}_\delta=\hat{\rho}+\delta. 
 $
Similarly,  \eqref{alpha} implies $\mu_\delta=\hat{\beta}_\delta$. Therefore,
 \begin{equation}\label{inqulitytildalambda}
   \mu_\delta=\hat{\beta}_\delta=\hat{\rho}+\delta. 
 \end{equation}

For any $\delta\geq 0$, we define $T$-periodic function $ \psi_\delta\in C^{2,1}(\mathbb{R}\times[0,T])$ by
  \begin{equation}\label{underlinevarphi2}
 \psi_\delta(x,t):=\alpha_\delta(t) e^{-\frac{\beta_\delta(t)}{2}x^2},
 \end{equation}
which, by definitions \eqref{beta} and \eqref{alpha}, solves
\begin{equation}\label{eq:hatphi}
  \mathcal{L}_\delta\psi_\delta:=
  \partial_t\psi_\delta-\partial_{xx}\psi_\delta-x(\rho(t)+\delta)\partial_{x}\psi_\delta=\mu_\delta\psi_\delta \quad\text{ in }\quad\mathbb{R}\times[0,T].
\end{equation}


 We first show $\liminf_{R\to\infty} \mu_R\geq \hat{\rho}$.
 For $\delta=0$, $\psi_0$ defined by \eqref{underlinevarphi2} is
 a super-solution
 to \eqref{unboundeigenvalue} in the sense of Definition \ref{appendixldef} for any $R>0$.
 By Proposition \ref{appendixprop},
 we have $\mu_R\geq \mu_0=\hat{\rho}$ for any $R>0$, and thus  $\liminf_{R\to\infty} \mu_R\geq \hat{\rho}$.

Next, we show $\limsup_{R\to\infty} \mu_R\leq \hat{\rho}$. Fix any $\delta>0$. Choose $R_\delta>0$ large such that $\delta x^2 \beta_\delta(t)\geq \hat{\rho}+\delta$ for all $|x|\geq R_\delta$ and $t\in[0,T]$.
Then let $\epsilon=\epsilon(\delta)>0$ be small so that  $\delta \psi_\delta(R_\delta, t)\geq \epsilon (\hat{\rho}+2\delta)$ for all $t\in[0,T]$.
Set $\tilde{\psi}_\delta=\max\{ \psi_\delta-\epsilon, 0\}$.
Note that we can choose
$\epsilon$ smaller if necessary such that
$\tilde{\psi}_\delta(x,t)>0$ holds for all $|x|\le R_\delta$ and $t\in [0, T]$.

On $\{(x,t): \tilde{\psi}_\delta(x,t)>0, |x|\ge R_\delta\}$, by \eqref{inqulitytildalambda} and \eqref{eq:hatphi} we calculate that
\begin{equation}\label{liu007-1}
\begin{split}
&\partial_t\tilde{\psi}_\delta-\partial_{xx}\tilde{\psi}_\delta- x\rho(t)\partial_x\tilde{\psi}_\delta-(\hat{\rho}+2\delta)\tilde{\psi}_\delta\\
=&(\hat{\rho}+\delta)\psi_\delta-\delta x^2 \beta_\delta(t) \psi_\delta -(\hat{\rho}+2\delta)(\psi_\delta-\epsilon)\\
\leq & (\hat{\rho}+\delta)\psi_\delta-\delta x^2 \beta_\delta(t) \psi_\delta\\
\leq & 0,
\end{split}
\end{equation}
where the last inequality follows from the choice of $R_\delta$.

On $\{(x,t): \tilde{\psi}_\delta(x,t)>0, |x|\le R_\delta\}$, we have
\begin{equation}\label{liu007-2}
\begin{split}
&\partial_t\tilde{\psi}_\delta-\partial_{xx}\tilde{\psi}_\delta- x\rho(t)\partial_x\tilde{\psi}_\delta-(\hat{\rho}+2\delta)\tilde{\psi}_\delta\\
=&(\hat{\rho}+\delta)\psi_\delta-\delta x^2 \beta_\delta(t) \psi_\delta -(\hat{\rho}+2\delta)(\psi_\delta-\epsilon)\\
\leq&(\hat{\rho}+\delta)\psi_\delta-(\hat{\rho}+2\delta)(\psi_\delta-\epsilon)
\\
\leq& \epsilon(\hat{\rho}+2\delta)
-\delta\psi_\delta(R_\delta, t)\\
\leq & 0,
\end{split}
\end{equation}
where the last inequality is due to the choice of $\epsilon$.

Choose some large $\tilde{R}_\delta>R_\delta$  such that  $\tilde{\psi}_\delta(\tilde{R}_\delta,t)=0$ for all $t\in[0,T]$.
By \eqref{liu007-1} and \eqref{liu007-2}, the constructed  $\tilde{\psi}_\delta$ is a sub-solution to \eqref{unboundeigenvalue} in the sense of Definition \ref{appendixldef} for any $R\geq \tilde{R}_\delta$.
A direct application of  Proposition \ref{appendixprop}  yields  $\mu_R\leq \hat{\rho}+2\delta$ for all $R\geq \tilde{R}_\delta$, and thus
$\limsup_{R\to\infty} \mu_R\leq \hat{\rho}+2\delta$. Letting $\delta\to 0$ completes the proof.
\end{proof}

\begin{prop}\label{sdlem2} For any $\kappa\in(0,1)$, suppose that
$$\begin{cases}
\partial_x m(x,t)<0,& (x,t)\in(0,\kappa)\times[0,T],\\
\partial_x m(x,t)>0,& (x,t)\in(\kappa,1)\times[0,T],\\
\partial_{xx} m(\kappa,t)\geq (\not\equiv) 0, &t\in[0,T],\\
m(x,0)=m(x,T),& x\in[0,1].
\end{cases}$$
Then we have
$$\lim_{D\rightarrow 0}\lambda (D)=\min\left\{ \hat{V}(0),\ \hat{V}(\kappa)+\partial_{xx} \hat{m}(\kappa),\ \hat{V}(1)\right\}.
$$
\end{prop}
\begin{proof} For any given $\epsilon>0$, we choose some small $\delta>0$ such that
\begin{equation}\label{defdelta}
\begin{array}{ll}
\medskip
|V(x,t)-V(0,t)|<\epsilon/2,\ &(x,t)\in[0,\delta]\times[0,T],\\
\medskip
|V(x,t)-V(\kappa,t)|<\epsilon/2,\ &(x,t)\in\left[\kappa-\delta,\kappa+\delta\right]\times[0,T],\\
|V(x,t)-V(1,t)|<\epsilon/2,\ &(x,t)\in[1-\delta,1]\times\in[0,T].
\end{array}
\end{equation}
\noindent {\bf Part I. } In this part, we establish the upper bound 
\begin{equation*}
  \limsup_{D\rightarrow 0}\lambda (D)\leq \lambda_{\min}:=\min\left\{ \hat{V}(0),\ \hat{V}(\kappa)+\partial_{xx} \hat{m}(\kappa),\ \hat{V}(1)\right\}.
\end{equation*}

By a similar argument as  in Proposition \ref{sdlem1}, it is straightforward to show that
$$\limsup_{D\rightarrow 0}\lambda (D)\leq\min\left\{ \hat{V}(0),\   \hat{V}(1)\right\}.$$
It remains to prove
\begin{equation}\label{upperlambda2.200}
  \limsup_{D\rightarrow 0}\lambda (D)\leq  \hat{V}(\kappa)+\partial_{xx}\hat{m}(\kappa). 
\end{equation}

Fix any $\epsilon>0$. For sufficiently small  $D$, we  construct a sub-solution $\underline \varphi$ such that
 \begin{equation}\label{ldnprop2.200}
 \begin{cases}
 \smallskip
 \mathcal{L}_D\underline{\varphi}\leq\left[\hat{V}(\kappa)+\partial_{xx}\hat{m}(\kappa)+2\epsilon\right]\underline{\varphi}
 &\text {in}\,\,((0,1)\setminus\mathbb{X})\times[0,T],\\
 \smallskip
 \partial_x\underline{\varphi}(0,t)=
 \partial_x\underline{\varphi}(1,t)=0  &\text {on}\,\,[0,T], \\
 \underline{\varphi}(x,0)=\underline{\varphi}(x,T) &\text{on}\,\,(0,1),
 \end{cases}
\end{equation}
 where the set $\mathbb{X}$ will be determined later.

To this end, we define 
 \begin{equation*}
   \bar{m}(x,t):=[\partial_{xx}m(\kappa,t)+\epsilon]\cdot\tfrac{\left(x-\kappa\right)^2}{2},
 \end{equation*}
and further choose $\delta$  smaller if necessary such that
 \begin{equation}\label{hatm}
    \begin{array}{c}
       |\partial_x\bar{m}|\geq|\partial_xm|~\,\,~\text{ in }~\left[\kappa-\delta,\kappa+\delta\right]\times[0,T].
     \end{array}
 \end{equation}

Let  $\bar\lambda_D$ denote the principal eigenvalue of the problem
\begin{equation}\label{auxiliaryproblem2.200}
\begin{cases}
\smallskip
\partial_t\psi-D\partial_{xx}\psi-\partial_x\bar{m}\partial_{x}\psi=\bar\lambda_D\psi  & \text{in}\,\,(\kappa-\delta,\kappa+\delta)\times[0,T],\\
\psi(\kappa-\delta,t)=\psi(\kappa+\delta,t)=0  &\text {on}\,\,[0,T], \\
 \psi(x,0)=\psi(x,T)  &\text{on}\,\,[\kappa-\delta,\kappa+\delta],
 \end{cases}
 \end{equation}
and the corresponding eigenfunction $\underline\psi_D$ is chosen to be positive in $(\kappa-\delta,\kappa+\delta)\times[0,T]$. Under the scaling $y=\frac{x-\kappa}{\sqrt{D}}$, we set $\underline\varphi_D(y,t):=\underline\psi_D(\sqrt{D}y+\kappa,t)$, which is the principal eigenfunction (associated to $\bar\lambda_D$) of the  problem
 \begin{equation*}
\begin{cases}
\smallskip
\partial_t\varphi-\partial_{yy}\varphi- y[\partial_{xx}m(\kappa,t)+\epsilon]\partial_y\varphi=\bar\lambda_D \varphi &\text {in}\,\,\left(-\frac{\delta}{\sqrt{D}},\frac{\delta}{\sqrt{D}}\right)\times[0,T],\\
\smallskip
\varphi(-\frac{\delta}{\sqrt{D}},t)=\varphi(\frac{\delta}{\sqrt{D}},t)=0  &\text {on}\,\,[0,T], \\
 \varphi(x,0)=\varphi(x,T)  &\text{on}\,\,\left[-\frac{\delta}{\sqrt{D}},\frac{\delta}{\sqrt{D}}\right].
\end{cases}
\end{equation*}
By Lemma \ref{wholespaceeigenvalue0}, we deduce that
 \begin{equation}\label{limithatlambda1}
   \lim_{D\to 0}\bar\lambda_D=\partial_{xx}\hat{m}(\kappa)+\epsilon. 
 \end{equation}

We extend $\underline\psi_D$, the principal eigenfunction of \eqref{auxiliaryproblem2.200}, to $[0,1]\times[0,T]$ by setting
$$\underline\psi_D\equiv0 \quad \text{on} \quad ([0,\kappa-\delta]\cup[\kappa+\delta,1])\times[0,T].$$
Applying the Hopf boundary lemma to \eqref{auxiliaryproblem2.200}, we have
$$\partial_{x}\underline\psi_D\left((\kappa-\delta)^+,\cdot\right)>0=\partial_{x}\underline\psi_D\left((\kappa-\delta)^-,\cdot\right),$$ 
$$\partial_{x}\underline\psi_D\left((\kappa+\delta)^+,\cdot\right)=0>\partial_{x}\underline\psi_D\left((\kappa+\delta)^-,\cdot\right),$$
so that we choose $\mathbb{X}$ by
$
\mathbb{X}=\{\kappa\pm\delta\}$.

Define 
 \begin{equation*}
 \underline{\varphi}(x,t)=f_\kappa(t)\cdot\underline\psi_D(x,t) \quad \text{in} \quad[0,1]\times[0,T],
 \end{equation*}
where $f_\kappa(t)$ is given by \eqref{deff} with $x=\kappa$.
We verify that $\underline{\varphi}$  satisfies \eqref{ldnprop2.200}.
By properties of $\underline\psi_D$ and \eqref{hatm} we can derive that
$$-\partial_xm\partial_{x}\underline\psi_D\leq-\partial_x\bar m\partial_{x}\underline\psi_D\quad\text{ in }\quad [0,1]\times[0,T].$$
Hence, direct calculations on $((0,1)\setminus\mathbb{X})\times[0,T]$ 
give
 \begin{equation*}
\begin{split}
\mathcal{L}_D\underline{\varphi}=&\left[-V\left(\kappa,t\right)+\hat{V}\left(\kappa\right)+V(x,t)\right]\underline{\varphi}+\left[\partial_{t}\underline\psi_D-D\partial_{xx}\underline\psi_D-\partial_xm\partial_{x}\underline\psi_D\right]f_{\kappa}(t)\\
\leq & \left[\hat{V}\left(\kappa\right)+\epsilon/2\right]\underline{\varphi}+\left[\partial_{t}\underline\psi_D-D\partial_{xx}\underline\psi_D-\partial_x\bar m\partial_{x}\underline\psi_D\right]f_{\kappa}(t)\\
= &\left[\hat{V}\left(\kappa\right)+\bar\lambda_D+\epsilon/2\right]\underline{\varphi}\\
\leq& \left[\hat{V}\left(\kappa\right)+\partial_{xx}\hat{m}(\kappa)+2\epsilon\right]\underline{\varphi},
\end{split}
\end{equation*}
provided that $D$ is small enough, where the last inequality  is a consequence of  \eqref{limithatlambda1}. Therefore,  $\underline \varphi$ defines a sub-solution satisfying \eqref{ldnprop2.200}, 
which together with  Proposition \ref{appendixprop} implies \eqref{upperlambda2.200}.
\medskip

\noindent {\bf Part II. } We shall establish the lower bound
\begin{equation}\label{lowerbound2.200}
  \liminf_{D\rightarrow 0}\lambda (D)\geq \lambda_{\min}:=\min\left\{ \hat{V}(0),\ \hat{V}(\kappa)+\partial_{xx} \hat{m}(\kappa),\ \hat{V}(1)\right\}. 
\end{equation}

For each small $\epsilon>0$, the main ingredient in the proof is to construct a positive continuous super-solution $\overline\varphi$  in the sense of  Definition \ref{appendixldef}, i.e. for sufficiently small $D$,
 \begin{equation}\label{supersoltion}
 \begin{cases}
  \smallskip
 \mathcal{L}_D\overline{\varphi}\geq (1-\epsilon)(\lambda_{\min}-\epsilon)\overline{\varphi}  &\text {in}\,\,((0,1)\setminus\mathbb{X})\times[0,T],\\
 \smallskip
 \partial_x\overline{\varphi}(0,t)=\partial_x\overline{\varphi}(1,t)=0  &\text {on}\,\,[0,T], \\
 \overline{\varphi}(x,0)=\overline{\varphi}(x,T) &\text{on}\,\,(0,1),
 \end{cases}
\end{equation}
 where the point set $\mathbb{X}$ will be determined in Step 3. Then \eqref{lowerbound2.200} follows from  Proposition \ref{appendixprop} and arbitrariness of $\epsilon$.


\medskip

\noindent {\bf Step 1.} We prepare some notations. 
First, we  choose suitable $T$-periodic function $\underline{\rho}(t)\geq\not\equiv0$ and small $\delta>0$ such that
\begin{equation}\label{defhatm1}
 \begin{cases}
\smallskip
\max\{\partial_{xx}m(\kappa,t)-\epsilon,0\}\leq\underline{\rho}(t)\leq\partial_{xx}m(\kappa,t), \quad t\in[0,T],\\
\underline{\rho}(t)|x-\kappa|\leq |\partial_x m(x,t)|, \quad (x,t)\in[\kappa-\delta,\kappa+\delta]\times[0,T].
\end{cases}
\end{equation}
Due to $\hat{\underline{\rho}}>0$, 
define   $r(t)$ as the unique positive $T$-periodic solution of
 \begin{equation}\label{r(t)}
 \frac{\dot{r}(t)}{2-\ell}=r(t)\left[\underline{\rho}(t)-\left(\frac{4}{(2-\ell)^2}+\frac{\epsilon}{2}\right)r(t)\right], 
 \end{equation}
where the small parameter $\ell\in(0,\epsilon/2]$ can be specified as follows: Note that there exist $0<\underline{r}<\overline{r}$ independent of $\ell\in(0,\epsilon/2]$ such that
$$0<\underline{r}< r(t)< \overline{r}\quad \text{ for all }t\in[0,T] \text{ and } \ell\in[0,\epsilon/2].$$
We fix $\ell\in(0,\epsilon/2]$ small such that
 \begin{equation}\label{conditionofell}
  \begin{cases}
  \smallskip
 \frac{\frac{2}{2-\ell}}{\frac{4}{(2-\ell)^2}+\frac{\epsilon}{2}}\geq 1-\epsilon,\\
 \nu:=\frac{\ell \overline{r}}{2-\ell}< \left[\sqrt{\frac{4}{(2-\ell)^2}+\frac{\epsilon}{2}}-1\right]\underline{r}.
 \end{cases}
 \end{equation}
Without loss of generality, we assume there is some $n_*\in \mathbb{N}$ ($n_*>3$) such that
 \begin{equation}\label{n*}
  1/\ell= 2^{n_*-2},
 \end{equation}
 and further choose $\delta$  smaller if necessary  such that
 $$\delta<\kappa-(n_*+1)\delta<\kappa+(n_*+1)\delta<1-\delta.$$

For fixed $r(t)$ and $\ell$, we define $(\alpha_1(t),\lambda_\ell)$ as the eigenpair of
\begin{equation}\label{r(t)1}
  \begin{cases}
  \smallskip
  \dot{\alpha}_1(t)+\frac{2}{2-\ell}\alpha_1(t)r(t)=\lambda_\ell\alpha_1(t) &\text{ in }[0,T],\\
 \alpha_1(0)=\alpha_1(T).
 \end{cases}
 \end{equation}
 Similar to \eqref{inqulitytildalambda}, we deduce from \eqref{r(t)} and \eqref{r(t)1} that
 $$\lambda_\ell=\frac{2}{2-\ell}\hat{r}=\frac{\frac{2}{2-\ell}}{\frac{4}{(2-\ell)^2}+\frac{\epsilon}{2}}\hat{\underline{\rho}},$$
which, together with \eqref{conditionofell}, leads to
\begin{equation}\label{relationlambdaell}
  \lambda_\ell\geq(1-\epsilon)\hat{\underline{\rho}}.
\end{equation}

\smallskip

\noindent {\bf Step 2.} We construct a positive super-solution $\overline\psi\in C(\mathbb{R}\times[0,T])$ for the auxiliary problem
\begin{equation}\label{wholespaceeigenvalue}
\begin{cases}
\smallskip
\partial_t\psi-\partial_{yy}\psi- y\underline{\rho}(t)\partial_y\psi=(1-\epsilon)\hat{\underline{\rho}} \psi &\text {in}\,\,\mathbb{R}\times[0,T],\\
\psi(x,0)=\psi(x,T) &\text{on}\,\,\mathbb{R}.
\end{cases}
\end{equation}

Using the notations introduced in Step 1, we define 
\begin{equation}\label{overlinepsi}
\overline{\psi}(y,t):=
\begin{cases}
 \smallskip
\alpha_1(t)e^{-\frac{r(t)}{2-\ell}y^2} &\mbox{ on } [-y_0,y_0]\times[0,T],\\
 \smallskip
\eta_1(t)e^{-\frac{(r(t)+\nu)y_0^\ell }{2-\ell}y^{2-\ell}} &\mbox{ on } (y_0,\infty)\times[0,T],\\
\eta_1(t)e^{-\frac{(r(t)+\nu)y_0^\ell }{2-\ell}(-y)^{2-\ell}} &\mbox{ on } (-\infty,-y_0)\times[0,T],\\
 \end{cases}
\end{equation}
where $y_0$ is a constant to be determined later, and $\eta_1(t)=\alpha_1(t)e^{\frac{\nu y_0^2}{2-\ell}}$, so that $\overline{\psi}\in C(\mathbb{R}\times[0,T])$ and $(\log\eta_1)'=(\log\alpha_1)'$ independent of $y_0$.

By the definition of $\nu$ in \eqref{conditionofell}, we may assert that for any $y_0>0$,
\begin{equation}\label{relationy0}
  \partial_y(\log\overline{\psi})(y_0^-,\cdot)=\left[-\frac{2r(\cdot)}{2-\ell}\right]y_0>\left[-(r(\cdot)+\nu)\right]y_0=\partial_y(\log\overline{\psi})(y_0^+,\cdot),
\end{equation}
and similarly, $\partial_y\overline{\psi}((-y_0)^-,\cdot)>\partial_y\overline{\psi}((-y_0)^+,\cdot)$. Therefore, in view of \eqref{relationlambdaell}, to verify that $\overline\psi$ defined by \eqref{overlinepsi} is a super-solution of \eqref{wholespaceeigenvalue}, it remains to choose large $y_0$ such that
\begin{equation}\label{supersolutioncondition}
  \partial_t\overline{\psi}-\partial_{yy}\overline{\psi}- y\underline{\rho}(t)\partial_y\overline{\psi}\geq \lambda_\ell \overline{\psi} \quad \text {in }\,\,(\mathbb{R}\setminus\{\pm y_0\})\times[0,T],
\end{equation}
which can be verified by the following computations:
\begin{itemize}
  \item [{\rm(i)}] For $y\in(-y_0,y_0)$, by \eqref{r(t)} and \eqref{overlinepsi}, direct calculations yield
   \begin{equation*}
\begin{split}
&\partial_t\overline{\psi}-\partial_{yy}\overline{\psi}- y\underline{\rho}(t)\partial_y\overline{\psi}\\
=& \left[(\log\alpha_1)'-\frac{\dot{r}(t)}{2-\ell}y^2+ \frac{2r(t)}{2-\ell}-\frac{4r^2(t)}{(2-\ell)^2}y^2+\frac{2r(t)\underline{\rho}(t)}{2-\ell}y^2 \right]\overline{\psi}\\
\geq& \left[(\log\alpha_1)'+ \frac{2r(t)}{2-\ell}\right]\overline{\psi}+\left[-\frac{\dot{r}(t)}{2-\ell}-\frac{4r^2(t)}{(2-\ell)^2}+r(t)\underline{\rho}(t) \right]y^2\overline{\psi}\\
\geq& \left[(\log\alpha_1)'+ \frac{2r(t)}{2-\ell}\right]\overline{\psi}\\
=&\lambda_\ell\overline{\psi};
\end{split}
\end{equation*}
  \item [{\rm(ii)}] For $y\in(y_0,\infty)$, again by \eqref{r(t)} and \eqref{overlinepsi}, we calculate that
   \begin{equation*}
\begin{split}
&\partial_t\overline{\psi}-\partial_{yy}\overline{\psi}- y\underline{\rho}(t)\partial_y\overline{\psi}-\lambda_\ell\overline{\psi}\\
= &\left[(\log\eta_1)'-\lambda_\ell \right]\overline{\psi}+y_0^\ell y^{2-\ell}\left[-\frac{\dot{r}(t)}{2-\ell}+ (1-\ell)(r(t)+\nu) y^{-2}\right]\overline{\psi}\\
&+y_0^\ell y^{2-\ell}\left[-(r(t)+\nu)^2y_0^{\ell} y^{-\ell}+(r(t)+\nu)\underline{\rho}(t)\right]\overline{\psi}\\
\geq& \left[(\log\alpha_1)'-\lambda_\ell \right]\overline{\psi}+y_0^\ell y^{2-\ell}\left[-\frac{\dot{r}(t)}{2-\ell}-(r(t)+\nu)^2+r(t)\underline{\rho}(t)\right]\overline{\psi}\\
=&\left[(\log\alpha_1)'-\lambda_\ell \right]\overline{\psi}+y_0^\ell y^{2-\ell}\left[\left(\frac{4}{(2-\ell)^2}+\frac{\epsilon}{2}\right)r^2(t)-(r(t)+\nu)^2\right]\overline{\psi}.
\end{split}
\end{equation*}
In light of $\left(\frac{4}{(2-\ell)^2}+\frac{\epsilon}{2}\right)r^2(t)>(r(t)+\nu)^2$ (due to \eqref{conditionofell}), we may pick  $y_0$ large enough to ensure  \eqref{supersolutioncondition} on $(y_0,\infty)\times[0,T]$;
 \item [{\rm(iii)}] For $y\in(-\infty,-y_0)$, we can verify \eqref{supersolutioncondition} by the same argument as in {\rm(ii)}.
\end{itemize}

Consequently, \eqref{supersolutioncondition} holds true, and $\overline\psi$ constructed by \eqref{overlinepsi} is a super-solution of \eqref{wholespaceeigenvalue} in the sense of Definition \ref{appendixldef}.

In what follows, we  divide the construction of super-solution  $\overline\varphi$ which satisfies \eqref{supersoltion}  into the following several steps via separating different regions; see Fig.\ref{figure1} for the profile of  $\overline\varphi$ to be constructed.
\begin{figure}[http!!]
  \centering
\includegraphics[height=3in]{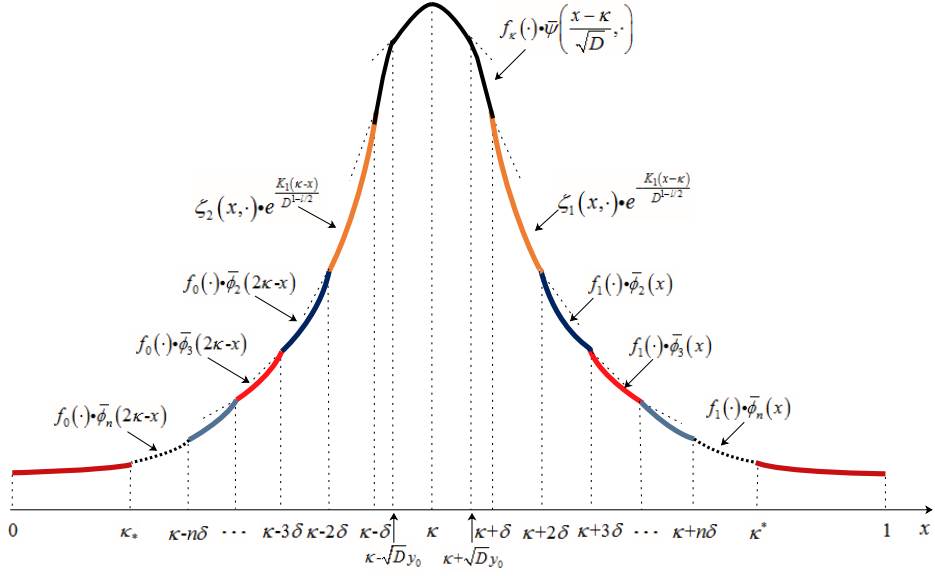}
  \caption{The profile of  $\overline\varphi$ for fixed $t\in[0,T]$.}\label{figure1}
  \end{figure}

\medskip
\noindent {\bf Step 3.} We construct super-solution $\overline\varphi$ on $[\kappa-\delta,\kappa+\delta]\times[0,T]$ satisfying \eqref{supersoltion}.
Let $\overline\psi$ be given by \eqref{overlinepsi} with fixed $y_0$ chosen in Step 2. 
We assume $\sqrt{D}y_0<\delta$, and define $\mathbb{X}$ by
 \begin{equation}\label{mathbbX}
  \mathbb{X}=\bigcup\limits_{n=1}^{n_*} \{\kappa\pm n\delta\}\cup\{\delta,1-\delta\} \cup\left\{\kappa\pm\sqrt{D}y_0\right\}.
 \end{equation}
where $n_*$ is chosen in \eqref{n*}. Set
\begin{equation}\label{overlinevarphi100}
\overline{\varphi}(x,t):=f_\kappa(t)\cdot\overline\psi\left(\frac{x-\kappa}{\sqrt{D}},t\right) \quad \text{on}\quad [\kappa-\delta,\kappa+\delta]\times[0,T],
\end{equation}
where  $f_\kappa(t)$ is defined by \eqref{deff} with $x=\kappa$. Note that  $\overline{\varphi}$ is symmetric in $x$ 
with respect to $x=\kappa$, and is decreasing  in $x$ for $x\geq \kappa$ and $t\in[0,T]$. Thus by \eqref{defhatm1} and \eqref{overlinevarphi100} we arrive at
\begin{equation}\label{liu000011}
-\partial_x m\partial_x \overline{\varphi}=|\partial_x m|\cdot|\partial_x \overline{\varphi}|\geq f_\kappa(t) \underline{\rho}(t)\frac{|x-\kappa|}{\sqrt{D}}\left|\partial_y\overline\psi\left(\frac{x-\kappa}{\sqrt{D}},t\right)\right|=-f_\kappa(t) \underline{\rho}(t)\cdot y\partial_y\overline\psi\left(y,t\right),
\end{equation}
where $y=\frac{x-\kappa}{\sqrt{D}}$. This  implies that on  $([\kappa-\delta,\kappa+\delta]\setminus\{\kappa\pm\sqrt{D}y_0\})\times[0,T]$, 
\begin{equation*}
\begin{split}
\mathcal{L}_D\overline{\varphi}\geq&\left[-V\left(\kappa,t\right)+\hat{V}\left(\kappa\right)+V(x,t)\right]\overline{\varphi}+\left[\partial_{t}\overline{\psi}-\partial_{yy}\overline{\psi}- \underline{\rho}(t)y\partial_y\overline\psi\right]f_{\kappa}(t)\\
\geq &\left[\hat{V}\left(\kappa\right)-\epsilon/2+(1-\epsilon)\hat{\underline{\rho}}\right]\overline{\varphi}\\
\geq & \left[\hat{V}(\kappa)+(1-\epsilon)\partial_{xx}\hat{m}(\kappa)-\epsilon(1-\epsilon)\right]\overline{\varphi}\\
\geq& (1-\epsilon)(\lambda_{\min}-\epsilon)\overline{\varphi},
\end{split}
\end{equation*}
where the first inequality is due to \eqref{liu000011},  the second inequality follows from \eqref{defdelta} and the fact that $\overline\psi$ is a super-solution of \eqref{wholespaceeigenvalue} (see Step 2), and the third inequality follows from \eqref{defhatm1}.

On the other hand, by \eqref{relationy0}, we have
$$\partial_x(\log\overline{\varphi})((\kappa+\sqrt{D}y_0)^+,\cdot)<\partial_x(\log\overline{\varphi})((\kappa+\sqrt{D}y_0)^-,\cdot)\quad\quad(\text{as}\quad \kappa+\sqrt{D}y_0\in\mathbb{X}).$$
Therefore,  $\overline{\varphi}$ defined by \eqref{overlinevarphi100}  satisfies \eqref{supersoltion} on  $[\kappa-\delta,\kappa+\delta]\times[0,T]$.

\medskip
\noindent {\bf Step 4.} We construct  $\overline\varphi$ which satisfies \eqref{supersoltion} on $(\kappa+\delta, \kappa+2\delta]\times[0,T]$.
Since $\sqrt{D}y_0<\delta$, by \eqref{overlinevarphi100} in Step 3 and  \eqref{overlinepsi} in Step 2, we have
\begin{equation}\label{relation3}
\begin{cases}
\smallskip
\log\overline{\varphi}(\kappa+\delta,t)= \log f_\kappa(t)+\log \eta_1(t)-\frac{(r(t)+\nu)y_0^\ell \delta^{2-\ell}}{(2-\ell)D^{1-\ell/2}},\\
  \partial_x(\log\overline{\varphi})((\kappa+\delta)^-,\cdot)=-(r(\cdot)+\nu) \frac{y_0^\ell\delta^{1-\ell}}{D^{1-\ell/2}},
\end{cases}
\end{equation}
whence there is some constant $K_0>0$ such that
\begin{equation}\label{relation4}
  \left|\partial_t(\log\overline{\varphi})(\kappa+\delta,\cdot)\right|=\left|(\log f_\kappa)'+(\log \eta_1)' 
  -\frac{\dot{r}(t)y_0^\ell \delta^{2-\ell}}{(2-\ell)D^{1-\ell/2}}\right|<\frac{K_0}{D^{1-\ell/2}}.
\end{equation}
We introduce a small parameter $\epsilon_0>0$ such that
 $$|\partial_x m|\geq \epsilon_0 \quad\text{on} \quad ([\delta,\kappa-\delta]\cup[\kappa+\delta,1-\delta])\times[0,T],$$
  and  fix constant $K_1$ so that
\begin{equation*}
  K_1>(\bar{r}+\nu)y_0^\ell\delta^{1-\ell}+2K_0/\epsilon_0.
\end{equation*}
Then we define
\begin{equation}\label{overlinevarphi101}
\overline{\varphi}(x,t):=\zeta_1(x,t)\cdot e^{-\frac{K_1(x-\kappa)}{D^{1-\ell/2}}}\quad \text{ on }\,\, (\kappa+\delta, \kappa+2\delta]\times[0,T].
\end{equation}
Here $\zeta_1\in C^{2,1}((\kappa+\delta, \kappa+2\delta)\times[0,T])$ is determined by
\begin{equation}\label{defeta}
\begin{array}{l}
\log\zeta_1(x,t)=\left[\frac{(\kappa+2\delta)-x}{\delta}\right]\cdot\left[\frac{K_1\delta}{D^{1-\ell/2}}+\log\overline\varphi(\kappa+\delta,t)\right]+\left[\frac{x-(\kappa+\delta)}{\delta}\right]\log f_1(t),
\end{array}
\end{equation}
with $T$-periodic function $f_1(t)$ defined in \eqref{deff} with $x=1$, so that
$$\zeta_1(\kappa+\delta,t)=e^{\frac{K_1\delta}{D^{1-\ell/2}}}\cdot\overline\varphi(\kappa+\delta,t). 
  $$
This implies immediately that $\overline{\varphi}$ defined by \eqref{overlinevarphi101} is continuous at $\{\kappa+\delta\}\times[0,T]$. In light of $\partial_x\zeta_1<0$ (for small $D$),  using \eqref{relation3} and \eqref{overlinevarphi101},   by choice of $K_1$ we can verify that
$$\partial_x(\log\overline{\varphi})((\kappa+\delta)^+,\cdot)<-K_1/D^{1-\ell/2}<\partial_x(\log\overline{\varphi})((\kappa+\delta)^-,\cdot)\quad\quad(\text{as}\quad \kappa+\delta\in\mathbb{X}).$$

  On the other hand, combined with \eqref{relation3}, \eqref{relation4}, and \eqref{defeta}, it is easily seen that
\begin{equation*}
\left|\partial_t(\log\zeta_1)\right|<\frac{2K_0}{D^{1-\ell/2}}, \quad -\frac{3K_1}{D^{1-\ell/2}}<\partial_x(\log\zeta_1)<0, \quad \text{and}\quad \partial_{xx}(\log\zeta_1)=0
\end{equation*}
for small $D$, and thus
$$\left|\frac{\partial_{xx}\zeta_1}{\zeta_1}\right|=\left|\partial_{xx}(\log\zeta_1)+[\partial_{x}(\log\zeta_1)]^2\right|<\frac{9K_1^2}{D^{2-\ell}},$$
from which, using \eqref{overlinevarphi101} and $-\partial_xm\cdot\partial_x(\log\zeta_1)\geq0$,
we may calculate that
\begin{equation*}
\begin{split}
\mathcal{L}_D\overline{\varphi}=&\left\{\partial_t(\log\zeta_1)-D\left[\frac{\partial_{xx}\zeta_1}{\zeta_1}-\frac{2K_1}{D^{1-\ell/2}}\partial_x(\log\zeta_1)+\frac{K_1^2}{D^{2-\ell}}\right]\right\}\overline{\varphi}\\
&+\left\{-\partial_xm\cdot\left[\partial_x(\log\zeta_1)-\frac{K_1}{D^{1-\ell/2}}\right]+V\right\}\overline{\varphi}\\
\geq &\left[-\frac{2K_0}{D^{1-\ell/2}}-\frac{16K_1^2}{D^{1-\ell}}+\frac{\epsilon_0K_1}{D^{1-\ell/2}}+V\right]\overline{\varphi}\\
=&\frac{1}{D^{1-\ell/2}}\left[-2K_0+\epsilon_0K_1-16K_1^2D^{\ell/2}+D^{1-\ell/2}V\right]\overline{\varphi}.
\end{split}
\end{equation*}
Since $\epsilon_0K_1>2K_0$ (by the definition of $K_1$),  we may choose $D$  small such that \eqref{supersoltion} holds.
Step 4 is thereby  completed.

\medskip
\noindent {\bf Step 5.} We construct $\overline\varphi$ on $(\kappa+2\delta, \kappa+3\delta]\times[0,T]$.
By  \eqref{overlinevarphi101} and \eqref{defeta}  in Step 4, we have
\begin{equation}\label{kappa+2delta}
  \log\overline{\varphi}(\kappa+2\delta,t)=\log f_1(t)-\frac{2K_1\delta}{D^{1-\ell/2}}\quad\text{and}\quad\partial_x(\log\overline{\varphi})((\kappa+2\delta)^-,\cdot)>-\frac{4K_1}{D^{1-\ell/2}}.
\end{equation}
Fix a constant $K_2$ such that $K_2>16K_1^2/\epsilon_0$, where $\epsilon_0$ is given in Step 4 such that $\partial_x m\geq \epsilon_0$ on $[\kappa+\delta,1-\delta]\times[0,T]$. Define 
\begin{equation}\label{overlinevarphi102}
\overline{\varphi}(x,t):=f_1(t)\cdot\overline{\phi}_2(x)\quad \text{on }\,\,(\kappa+2\delta, \kappa+3\delta]\times[0,T],
\end{equation}
where $\overline{\phi}_2$ solves
\begin{equation}\label{overlinephi}
\begin{cases}
\smallskip
(\log\overline{\phi}_2)'(x)=-\frac{4K_1}{D^{1-\ell/2}}\left[\frac{\kappa+3\delta-x}{\delta}\right]-\frac{K_2}{D^{1-\ell}}\left[\frac{x-(\kappa+2\delta)}{\delta}\right]\quad \text{in }\,\,(\kappa+2\delta, \kappa+3\delta],\\
 \log\overline{\phi}_2(\kappa+2\delta)=-\frac{2K_1\delta}{D^{1-\ell/2}}.
\end{cases}
\end{equation}
Together \eqref{overlinevarphi102} with \eqref{kappa+2delta} and \eqref{overlinephi}, we discover that $\overline{\varphi}$ is continuous at $\{\kappa+2\delta\}\times[0,T]$, and
$$\partial_x(\log\overline{\varphi})((\kappa+2\delta)^+,\cdot)=-4K_1/D^{1-\ell/2}<\partial_x(\log\overline{\varphi})((\kappa+2\delta)^-,\cdot)\quad(\text{as}\quad\kappa+2\delta\in\mathbb{X}).$$
For all $x\in(\kappa+2\delta, \kappa+3\delta]$, by \eqref{overlinephi} we  have
$$\left|\frac{\overline{\phi}_2''}{\overline{\phi}_2}\right|=\left|(\log\overline{\phi}_2)''+[(\log\overline{\phi}_2)']^2\right|\leq\frac{4K_1}{\delta D^{1-\ell/2}}+\frac{16K_1^2}{D^{2-\ell}},$$
from which we arrive at
\begin{equation*}
\begin{split}
\mathcal{L}_D\overline{\varphi}=&\left[(\log f_1)'-D\overline{\phi}_2''/\overline{\phi}_2-\partial_xm\cdot(\log\overline{\phi}_2)'+V\right]\overline{\varphi}\\
\geq &\left[(\log f_1)'-\frac{4K_1}{\delta}D^{\ell/2}-\frac{16K_1^2}{D^{1-\ell}}+\frac{\epsilon_0K_2}{D^{1-\ell}}+V\right]\overline{\varphi}. 
\end{split}
\end{equation*}
In view of $\epsilon_0K_2>16K_1^2$, we once more would select $D$  small enough such that \eqref{supersoltion} holds. 

\medskip
\noindent {\bf Step 6.} We construct $\overline\varphi$ on $(\kappa+3\delta, \kappa+(n_*+1)\delta]\times[0,T]$, where $n_*$ is determined by \eqref{n*} in Step 1. By the definition of $\overline{\phi}_2$ in \eqref{overlinephi}, we have
\begin{equation}\label{liu000018}
\partial_x(\log\overline{\varphi})\left((\kappa+3\delta)^-,\cdot\right)=(\log\overline{\phi}_2)'\left(\kappa+3\delta\right)=-K_2/D^{1-\ell}.
\end{equation}
We introduce a sequence $\{K_n\}_{n=3}^{n_*}$  independent of $D$ such that $K_{n}>K_{n-1}^2/\epsilon_0$.  With $\overline{\phi}_2$  in hand, by induction  we define $\overline{\phi}_n\in C^{2,1}([\kappa+n\delta,\kappa+(n+1)\delta])$ ($n=3,\ldots,n_*$) to solve
\begin{equation}\label{overlinephi_n}
\begin{cases}
\smallskip
(\log\overline{\phi}_{n})'(x)=-\left[\frac{K_{n-1}}{D^{1-2^{n-3}\ell}}+\epsilon\right]\cdot\left[\frac{\kappa+n\delta-x}{\delta}\right]-\frac{K_{n}}{D^{1-2^{n-2}\ell}}\left[\frac{x-(\kappa+n\delta)}{\delta}\right] \,\,\text{in }(\kappa+n\delta, \kappa+(n+1)\delta],\\
 \log\overline{\phi}_{n}(\kappa+n\delta)=\log\overline{\phi}_{n-1}(\kappa+n\delta).
\end{cases}
\end{equation}
Then we define
\begin{equation}\label{overlinevarphi103}
\overline{\varphi}(x,t):=f_1(t)\cdot\overline{\phi}_n(x)\quad \text{on}\quad (\kappa+n\delta, \kappa+(n+1)\delta]\times[0,T].
\end{equation}
By \eqref{liu000018} and \eqref{overlinephi_n}, it can be verified that 
\begin{equation*}
\begin{split}
\partial_x(\log\overline{\varphi})\left((\kappa+3\delta)^+,\cdot\right)=\partial_x(\log\overline{\phi}_3)\left(\kappa+3\delta\right)=-\left[\frac{K_2}{D^{1-\ell}}+\epsilon\right]<\partial_x(\log\overline{\varphi})\left((\kappa+3\delta)^-,\cdot\right),
\end{split}
\end{equation*}
and similarly for $4 \leq n\leq n_*$,
\begin{equation*}
\begin{split}
\partial_x(\log\overline{\varphi})\left((\kappa+n\delta)^+,\cdot\right)<\partial_x(\log\overline{\varphi})\left((\kappa+n\delta)^-,\cdot\right)\quad(\text{as}\,\,\kappa+n\delta\in\mathbb{X}).
\end{split}
\end{equation*}

For each $3\leq n\leq n_*$,  
it follows from \eqref{overlinephi_n} that for $x\in(\kappa+n\delta, \kappa+(n+1)\delta]$
\begin{equation}\label{liu000012}
-\left[\frac{K_{n-1}}{D^{1-2^{n-3}\ell}}+\epsilon\right]\leq(\log\overline{\phi}_n)'\leq -\frac{K_{n}}{D^{1-2^{n-2}\ell}},
\end{equation}
and then as in Step 5, we derive that
\begin{equation}\label{liu000013}
\left|\frac{\overline{\phi}_n''}{\overline{\phi}_n}\right|=\left|(\log\overline{\phi}_n)''+[(\log\overline{\phi}_n)']^2\right|\leq\frac{2K_{n-1}}{\delta D^{1-2^{n-3}\ell}}+\frac{K_{n-1}^2}{D^{2-2^{n-2}\ell}}.
\end{equation}
By \eqref{liu000012} and \eqref{liu000013},  on $(\kappa+n\delta, \kappa+(n+1)\delta]\times[0,T]$, we calculate that
\begin{equation*}
\begin{split}
\mathcal{L}_D\overline{\varphi}=&\left[(\log f_1)'- D\overline{\phi}_n''/\overline{\phi}_n-\partial_xm\cdot(\log\overline{\phi}_n)'+V\right]\overline{\varphi}\\
\geq &\left[(\log f_1)'-\frac{2K_{n-1}}{\delta}D^{2^{n-3}\ell}-\frac{K_{n-1}^2}{D^{1-2^{n-2}\ell}}+\frac{\epsilon_0K_{n}}{D^{1-2^{n-2}\ell}}+V\right]\overline{\varphi}.
\end{split}
\end{equation*}
 Since $\epsilon_0K_{n}>K_{n-1}^2$,  we choose $D$ to be small so that  $\overline{\varphi}$  satisfies \eqref{supersoltion}.

\medskip
\noindent {\bf Step 7.} We construct $\overline\varphi$ on $(\kappa+(n_*+1)\delta, 1]\times[0,T]$.
Set $\kappa^*=\kappa+(n_*+1)\delta$. 
Observe from Step 6 and the definition of $n_*$ in \eqref{n*} that
\begin{equation*}
  \partial_x(\log\overline{\varphi})((\kappa^*)^-,\cdot)=-K_{n_*}/D^{1-2^{n_*-2}\ell}=-K_{n_*}.
\end{equation*}

We define 
\begin{equation}\label{overlinevarphi104}
\overline{\varphi}(x,t):=f_1(t)\overline{\phi}_{n_*}(\kappa^*)\cdot
\begin{cases}
\medskip
 e^{-K_*(x-\kappa^*)} & \text{on } (\kappa^*,1-\delta]\times[0,T],\\
  e^{\frac{K_*+\epsilon}{2\delta}(1-x)^2+\theta_1} & \text{on } (1-\delta,1]\times[0,T],
\end{cases}
\end{equation}
where $K_*>K_{n_*}$ will  be determined later, and the parameter $\theta_1$ is chosen such that $\overline{\varphi}$ is continuous at $\{1-\delta\}\times[0,T]$. It 
follows that
$$\partial_x(\log\overline{\varphi})(x^+,\cdot)<\partial_x(\log\overline{\varphi})(x^-,\cdot)\quad\text{for}\quad x\in\{ \kappa^*,1-\delta\}\subset\mathbb{X}.$$

It remains to verify that $\overline{\varphi}$ defined by \eqref{overlinevarphi104} satisfies \eqref{supersoltion}.
For $x\in(\kappa^*,1-\delta]$, since $\partial_x m\geq\epsilon_0$, using \eqref{overlinevarphi104} we deduce that
\begin{equation*}
\begin{split}
\mathcal{L}_D\overline{\varphi}\geq\left[(\log f_1)'- DK_*^2+\epsilon_0 K_*+V\right]\overline{\varphi}. 
\end{split}
\end{equation*}
By choosing $K_*$  large and then choosing $D$  small, we see that $\overline{\varphi}$ satisfies \eqref{supersoltion}.

For $x\in(1-\delta,1]$, since $-\partial_xm\partial_x\overline{\varphi}\geq0$, by  \eqref{overlinevarphi104}, letting $D$ be so small that
 \begin{equation*}
\begin{split}
\mathcal{L}_D\overline{\varphi}\geq&\left\{(\log f_1)'- D(K_*+\epsilon)/\delta\cdot\Big[(K_*+\epsilon)(x-1)^2/\delta+1\Big]+V\right\}\overline{\varphi}\\
\geq &\left[\hat{V}(1)-V(1,t)+V(x,t)-\epsilon/2\right]\overline{\varphi}\\
\geq&(\lambda_{\min}-\epsilon)\overline{\varphi},
\end{split}
\end{equation*}
where the last inequality is  due to \eqref{defdelta}.

By Steps 3-7, we have already constructed the strict super-solution $\overline{\varphi}$ satisfying \eqref{supersoltion} on $[\kappa-\delta,1]\times[0,T]$ with the set $\mathbb{X}$ given by \eqref{mathbbX}, which is summarized  in the following table for the convenience of readers; see also Fig.\ref{figure1}.
 \begin{table}[!htbp]
\centering
\footnotesize
\begin{tabular}{|c|c|c|}
\hline
\cline{1-3}
\multicolumn{3}{|c|}{Construction of $\overline{\varphi}$ on $[\kappa-\delta,1]\times[0,T]$}\\
\hline
\cline{1-3}
$\overline{\varphi}(x,t)$ & Region &Defined ~in \\
\hline
\cline{1-3}
$f_\kappa(t)\cdot\overline\psi\left(\frac{x-\kappa}{\sqrt{D}},t\right)$ & $[\kappa-\delta,\kappa+\delta]\times[0,T]$ & \eqref{overlinevarphi100}  in Step 3 \\
\hline
\cline{1-3}
$\zeta_1(x,t)\cdot e^{-\frac{K_1(x-\kappa)}{D^{1-\ell/2}}}$ & $[\kappa+\delta,\kappa+2\delta]\times[0,T]$ & \eqref{overlinevarphi101} and \eqref{defeta} in Step 4 \\
\hline
\cline{1-3}
$f_1(t)\cdot\overline{\phi}_n(x)$ & \tabincell{c}{$(\kappa+n\delta, \kappa+(n+1)\delta]\times[0,T] $\\
 $(n=2,\ldots, n_*)$} & \tabincell{c}{\eqref{overlinephi_n} and \eqref{overlinevarphi103}  in \\
  Steps 5 and 6 }\\
 \hline
\cline{1-3}
$f_1(t)\overline{\phi}_{n_*}(\kappa^*)\cdot e^{-K_*(x-\kappa^*)}$ & $( \kappa^*,1-\delta]\times[0,T]$ & \eqref{overlinevarphi104}  in Step  7 \\
\hline
\cline{1-3}
$f_1(t)\overline{\phi}_{n_*}(\kappa^*)\cdot e^{\frac{K_*+\epsilon}{2\delta}(1-x)^2+\theta_1}$ & $(1-\delta,1]\times[0,T]$ & \eqref{overlinevarphi104}  in Step  7\\
\hline
\cline{1-3}
\end{tabular}
\end{table}

Finally, we construct  $\overline{\varphi}$ on $[0,\kappa-\delta)\times[0,T]$ symmetrically; and precisely, we define
\begin{equation}\label{overlinevarphi105}
  \overline{\varphi}(x,t)=
\begin{cases}
\zeta_2(x,t)\cdot e^{\frac{K_1(\kappa-x)}{D^{1-\ell/2}}} & \text{on } [\kappa-2\delta, \kappa-\delta)\times[0,T], \\
f_0(t)\cdot\overline{\phi}_n(2\kappa-x) &\text{on } \begin{cases} [\kappa-(n+1)\delta, \kappa-n\delta)\times[0,T],\\
(n=2,\ldots, n_*),\end{cases}\\
\medskip
f_0(t)\overline{\phi}_{n_*}(\kappa_*)\cdot e^{K_*(\kappa_*-x)} & \text{on }[\delta, \kappa_*)\times[0,T],\\
f_0(t)\overline{\phi}_{n_*}(\kappa_*)\cdot e^{\frac{K_*+\epsilon}{2\delta}x^2+\theta_2} & \text{on } [0,\delta)\times[0,T],
\end{cases}
\end{equation}
where $\kappa_*=\kappa-(n_*+1)\delta$, and similar to  \eqref{defeta}, $\zeta_2$ solves
\begin{equation*}
\begin{array}{l}
\log\zeta_2(x,t)=\left[\frac{x-(\kappa-2\delta)}{\delta}\right]\cdot\left[\frac{K_1\delta}{D^{1-\ell/2}}+\log\overline\varphi(\kappa-\delta,t)\right]+\left[\frac{(\kappa-\delta)-x}{\delta}\right]\log f_0(t),
\end{array}
\end{equation*}
with $f_0$ defined in \eqref{deff} with $x=0$, and $\theta_2$ is chosen such that $\overline{\varphi}$ is continuous at $\{\delta\}\times[0,T]$.
Using the same arguments as in Steps 4-7, we may conclude that $\overline{\varphi}$ defined by \eqref{overlinevarphi105} verifies \eqref{supersoltion}, and thus $\overline{\varphi}$ constructed above defines a super-solution  on the entire region $[0,1]\times[0,T]$ with $\mathbb{X}$ given by \eqref{mathbbX}. Therefore, \eqref{lowerbound2.200} follows from Proposition \ref{appendixprop}.
\end{proof}


By assuming $\partial_xm(0,t)>0$ and $\partial_xm(1,t)<0$ for each $t\in[0,T]$, it is shown in Proposition \ref{sdlem1}  that the limit of $\lambda(D)$ as $D\to 0$ does not depend upon the value of $V$ on  boundary points $\{0,1\}\times[0,T]$. However, without the positivity assumption of $\partial_xm(0,t)$, one can prove

\begin{lemma}\label{sdlem5}
Suppose that $\partial_x m(x,t)>0$ for all $(x,t)\in(0,1)\times[0,T]$.
Then
$$\lim_{D\rightarrow 0}\lambda (D)=\min\left\{\hat{V}(0)+[\partial_{xx}\hat{m}]_{+}(0),\,\, \hat{V}(1)\right\},$$
where $\partial_{xx}\hat{m}(0)$ is defined by \eqref{boundarycondition}.
\end{lemma}

\begin{proof}
If $\partial_x m(0,t)=0$ for all $t\in[0,T]$, Lemma \ref{sdlem5} can be proved directly by constructing the same super/sub-solutions
as those in Proposition \ref{sdlem2} defined on $[\kappa,1]\times[0,T]$. It suffices to consider the remaining case $\partial_x \hat{m}(0)>0$ and in view of $\partial_{xx}\hat{m}_{+}(0)=\infty$ in this case, i.e. to show
$$\lim_{D\rightarrow 0}\lambda (D)=\hat{V}(1).$$

First, similarly as in the proof of Proposition \ref{sdlem1}, we may construct a sub-solution to prove  $\limsup_{D\rightarrow 0}\lambda (D)\leq\hat{V}(1)$. In the sequel, we show
\begin{equation}\label{lowerbound11}
 \liminf_{D\rightarrow 0}\lambda (D)\geq\hat{V}(1).
\end{equation}

 For any given $\epsilon>0$, we fix some small  $\delta>0$  such that
$$
|V(x,t)-V(1,t)|<\epsilon/2\,\,\text{ on }\,\,[1-\delta,1]\times[0,T].$$
The strategy is to construct a positive super-solution $\overline \varphi\in C^{2,1}([0,1]\times[0,T])$, which satisfies
 \begin{equation}\label{supersoltionlem2.5}
 \begin{cases}
  \smallskip
 \mathcal{L}_D\overline{\varphi}\geq \left[\hat{V}(1)-\epsilon\right]\overline{\varphi} &\text {in}\,\,(0,1)\times[0,T],\\
 \smallskip
 \partial_x\overline{\varphi}(0,t)<0, \ \
 \partial_x\overline{\varphi}(1,t)=0  &\text {on}\,\,[0,T], \\
 \overline{\varphi}(x,0)=\overline{\varphi}(x,T) &\text{on}\,\,(0,1)
 \end{cases}
\end{equation}
for sufficiently small $D$. To this end, we proceed as follows: 

On $[1-\delta,1]\times[0,T]$,  we define 
$$\overline{\varphi}(x,t):=f_1(t)\cdot   e^{\frac{M_2}{2\delta}(1-x)^2}\quad \text{on} \quad [1-\delta,1]\times[0,T], 
$$
where  $M_2>0$ will be determined later, and $f_1(t)$ is given by \eqref{deff} with $x=1$.
As Step 2 in Proposition \ref{sdlem1}, one can verify that \eqref{supersoltionlem2.5} holds on $[1-\delta,1]\times[0,T]$.

On  $[0,\delta]\times[0,T]$, since $\partial_x m(0,t)\geq(\not\equiv)0$ for $t\in[0,T]$ (due to $\partial_x \hat{m}(0)>0$ ),  one can find some $t_0\in(0,T)$ and positive constants $\epsilon_0,\delta_0$ such that
$$\partial_x m(x,t)>\epsilon_0\quad \text{for any}\quad (x,t)\in[0,\delta]\times[t_0-\delta_0,t_0+\delta_0].$$ 
Fix $\eta_2\in C^\infty([0,T])$ to be  a positive $T$-periodic function such that
\begin{equation}\label{theta1}
 \left(\log\eta_2(t)\right)'>\|V(\cdot,t)\|_{L^{\infty}}+|\hat{V}(1)|\quad\text{for }\,\,t\in [0,t_0-\delta]\cup[t_0+\delta,T].
\end{equation}

We then define, for $(x,t)\in[0,\delta]\times [0,T]$,
$$\overline{\varphi}(x,t):=\eta_2(t)\cdot e^{-M_2x}.$$
On $[0,\delta]\times [t_0-\delta_0,t_0+\delta_0]$, since $\partial_x m(x,t)>\epsilon_0$, by  straightforward computations we deduce
\begin{equation*}
\begin{split}
\mathcal{L}_D\overline{\varphi}
\geq \Big[(\log\eta_2)'-D M^2_2+M_2 \epsilon_0-V\Big]\overline{\varphi},
\end{split}
\end{equation*}
whence by choosing $M_2$  large and then choosing $D$  small, we have $\mathcal{L}_D\overline{\varphi}\geq \hat{V}(1)\overline{\varphi}$.

On the other hand, on  $[0,\delta]\times\left( [0,t_0-\delta]\cup[t_0+\delta,T]\right)$, in view of \eqref{theta1} and $-\partial_xm \partial_xm\overline{\varphi}\geq 0$, by letting $D$ be small, we arrive at
\begin{equation*}
\begin{split}
\mathcal{L}_D\overline{\varphi}&\geq\Big[(\log\eta_2)'-D M^2_2-V\Big]\overline{\varphi}\geq \hat{V}(1)\overline{\varphi},
\end{split}
\end{equation*}
whence \eqref{supersoltionlem2.5} is  verified on $[0,\delta]\times [0,T]$.

On  $(\delta,1-\delta)\times[0,T]$, notice  from the definitions of $\overline{\varphi}$ above that $$\partial_{x}(\log\overline{\varphi})(\delta)=\partial_{x}(\log\overline{\varphi})(1-\delta)=-M_2.$$
We can always find $\overline{\varphi}\in C^{2,1}([\delta,1-\delta]\times[0,T])$ such that $\overline{\varphi}(\cdot,0)=\overline{\varphi}(\cdot,T)$ and
$$
\partial_{x}\log\overline{\varphi}\leq-M_2 \quad \text{and}\quad |\partial_t\log\overline{\varphi}|\leq 2
\left\||(\log f_1)'|+|(\log \eta_2)'|\right\|_{L^\infty},
$$
and then  \eqref{supersoltionlem2.5} can be verified directly  by further choosing   $M_2$ large and  $D$ small.

Therefore, such a super-solution $\overline{\varphi}$ defined above satisfies \eqref{supersoltionlem2.5}, and  Proposition \ref{appendixprop} concludes the proof.
\end{proof}

\begin{corollary}\label{coro0001}
Assume
$V(x,t)=V(x)$ and $\partial_xm(x,t)=m'(x)$. Suppose that $m'(x)>0$ for all $x\in(0,1)$. Then we have
$$\lim_{D\rightarrow 0}\lambda (D)=\min\left\{V(0)+[m'']_{+}(0),\,\, V(1)\right\}.$$
\end{corollary}
\begin{remark}
{\rm Corollary {\rm \ref{coro0001}} cannot be covered by Theorem {\rm \ref{independentt}}. 
It  also provides an example such that Theorem 1.2 in \cite{CL2012}  fails without the assumption  $|\nabla m|\neq 0$ on $\partial\Omega$ therein.}
\end{remark}

To establish Theorem \ref{thm1.1}, we  prepare the following
\begin{lemma}\label{sdlemdegen}
Given any $0\leq\underline{\kappa}<\overline{\kappa}\leq 1$, let $\lambda(D)$ be the principal eigenvalue of the problem
\begin{equation}\label{SD_eq 2}
\begin{cases}
\partial_t\varphi-D\partial_{xx}\varphi+V\varphi=\lambda(D) \varphi & \text{in }(\underline{\kappa},\overline{\kappa})\times [0,T],\\
c_1\partial_x\varphi(\underline{\kappa},t)-(1-c_1)\varphi(\underline{\kappa},t)=0 & \text{on } [0,T],\\
c_2\partial_x\varphi(\overline{\kappa},t)+(1-c_2)\varphi(\overline{\kappa},t)=0 & \text{on } [0,T],\\
\varphi(x,0)=\varphi(x,T) & \text{on } [\underline{\kappa},\overline{\kappa}],
\end{cases}
\end{equation}
where $c_1,c_2\in [0,1]$. Then we have
$$\lim_{D\rightarrow 0}\lambda (D)=\min_{x\in[\underline{\kappa},\overline{\kappa}]}\hat{V}(x).$$
\end{lemma}
\begin{remark}
{\rm Lemma {\rm\ref{sdlemdegen}} is proved in Lemma {\rm2.4(c)} of \cite{Hutson2001} for the case $c_1=c_2=1$.}
\end{remark}
\begin{proof}[Proof of Lemma {\rm\ref{sdlemdegen}}]
For the upper bound, it suffices to claim that
$$\limsup_{D\rightarrow 0}\lambda (D)\leq V(\tilde{x})\quad \text{for any}\,\, \tilde{x}\in (\underline{\kappa},\overline{\kappa}).$$
Indeed, we follow the ideas as in  Proposition \ref{sdlem1} and define a sub-solution
\begin{equation}\label{subsolutiondeg}
  \underline{\varphi}(x,t):=f_{\tilde{x}}(t)\cdot \tilde{z}(x)
\end{equation}
 with $f_{\tilde{x}}(t)$ defined in \eqref{deff} with $x=\tilde{x}$
and
$$\tilde{z}(x)=\begin{cases}
\cos \left( \frac{\pi}{2\tilde{\delta}}(x-\tilde{x})\right)
& \text{ on } [\tilde{x}-\tilde{\delta},\tilde{x}+\tilde{\delta}],\\
0,& \text{ on } [0,\tilde{x}-\tilde{\delta})\cup(\tilde{x}+\tilde{\delta},1].
\end{cases}$$
Here $\tilde{\delta}$ is chosen such that $|V(x,t)-V(\tilde{x},t)|<\epsilon/2$ in $[\tilde{x}-\tilde{\delta},\tilde{x}+\tilde{\delta}]\times[0,T]$ for any given $\epsilon>0$. One may verify readily that
\begin{equation*}
\mathcal{L}_D\underline{\varphi}\leq \left[\hat{V}(\tilde{x})+\epsilon\right]\underline{\varphi},
\end{equation*}
so that the upper bound follows from Proposition \ref{appendixprop}.

It remains to prove
\begin{equation}\label{liu000029}
\liminf_{D\rightarrow 0}\lambda (D)\geq\min_{x\in[\underline{\kappa},\overline{\kappa}]}\hat{V}(x).
\end{equation}
For any $\epsilon>0$,  we choose some  $T$-periodic function $V_\epsilon\in C^{2,1}([\underline{\kappa},\,\overline{\kappa}]\times[0,T])$ such that
$$\left\|V_\epsilon-V\right\|_{L^\infty([0,1]\times[0,T])}\leq \epsilon.$$
Then we define $T$-periodic function $\varphi_\epsilon$ by
\begin{equation}\label{liu000028}
\varphi_\epsilon(x,t):=\mbox{exp} \left[-\int_0^t V_\epsilon(x,s)\mathrm{d}s+t\hat{V_\epsilon}(x)\right]\beta_\epsilon (x),
\end{equation}
where  $\beta_\epsilon\in C^2([\underline{\kappa},\,\overline{\kappa}])$ is a positive function  and is chosen such that
\begin{equation}\label{liu000030}
c_1\partial_x\varphi_\epsilon(\underline{\kappa},t)- (1-c_1) \varphi_\epsilon(\underline{\kappa},t)\leq0 \quad\text{and}\quad c_2\partial_x\varphi_\epsilon(\overline{\kappa},t)+(1-c_2) \varphi_\epsilon(\overline{\kappa},t)\geq 0.
\end{equation}

By \eqref{liu000028} and the definition of $V_\epsilon$, we may choose  $D$ small to derive that
\begin{equation*}
\begin{split}
\partial_t\varphi_\epsilon-D\partial_{xx}\varphi_\epsilon+V\varphi_\epsilon&=\left[\hat{V}_\epsilon(x)-V_\epsilon(x,t)+V(x,t)\right]\varphi_\epsilon-D\partial_{xx}\varphi_\epsilon\\
&\geq \left[\min_{x\in[\underline{\kappa},\overline{\kappa}]}\hat{V}(x)-3\epsilon\right]\varphi_\epsilon,
\end{split}
\end{equation*}
which together with \eqref{liu000030} implies that $\varphi_\epsilon$  defined by \eqref{liu000028} is a super-solution in the sense of Definition \ref{appendixldef} with $\mathbb{X}=\emptyset$. Thus \eqref{liu000029} follows from Proposition \ref{appendixprop}, and the proof of Lemma \ref{sdlemdegen} is now complete.
\end{proof}

We are now in a position to prove Theorem \ref{thm1.1}.

\begin{proof}[Proof of Theorem {\rm\ref{thm1.1}}]
The proof can be carried out by the same ideas as in Propositions \ref{sdlem1} and \ref{sdlem2} with the help of Lemmas \ref{sdlem5} and \ref{sdlemdegen}. Here we just outline it for completeness.

\medskip
\noindent {\bf Step 1.} We establish the upper bound of $\limsup_{D\rightarrow 0}\lambda(D)$.
First, using a similar argument as in Lemma \ref{sdlemdegen}, one can  establish
$$
\limsup_{D\rightarrow 0}\lambda(D)\leq\min_{i\in\mathbf{B}}\left\{\min_{x\in[\kappa_i,\kappa_{i+1}]}\hat{V}(x)\right\} 
$$
by constructing a suitable sub-solution like \eqref{subsolutiondeg}. Similarly, the estimate
 $$
\limsup_{D\rightarrow 0}\lambda(D)\leq \min\left\{\hat{V}(0)+[\partial_{xx}\hat{m}]_+(0),\,\, \hat{V}(1)+[\partial_{xx}\hat{m}]_+(1)\right\}
$$
 can also be proved; the details are omitted here. It remains to show
 \begin{equation}\label{upperestimate}
   \limsup_{D\rightarrow 0}\lambda(D)\leq\hat{V}(\kappa_i)+[\partial_{xx}\hat{m}]_+(\kappa_i)\quad\text{ for all }1\leq i\leq N.
 \end{equation}

Fix any $\epsilon>0$.  Choose some small $\delta>0$ such that
$|V(x,t)-V(\kappa_i,t)|<\epsilon/2$ on $[\kappa_i-\delta,\kappa_i+\delta]\times[0,T]$ for all $1\leq i\leq N$.
To prove \eqref{upperestimate}, we define
$$\underline{\varphi}_i(x,t):=\begin{cases}\smallskip
f_{\kappa_i}(t)\cdot \underline{z}(x) & \text{ if } \partial_{xx}\hat{m}(\kappa_i)\leq0,\\
f_{\kappa_i}(t)\cdot\underline{\psi}_D(x,t) & \text{ if } \partial_{xx}\hat{m}(\kappa_i)>0,
\end{cases}$$
where $f_{\kappa_i}$ and $\underline{z}$ are defined respectively  by \eqref{deff} and \eqref{underlinez}, and $\underline{\psi}_D$ denotes the principal eigenfunction of \eqref{auxiliaryproblem2.200} with $\kappa=\kappa_i$. The same arguments as in Step 1 of Propositions \ref{sdlem1} and \ref{sdlem2}  allow us to verify
that such a function $\underline\varphi_i$ defines a sub-solution in the sense of Definition \ref{appendixldef} such that for sufficiently small  $D$,
 \begin{equation*}
 \begin{cases}
 \smallskip
 \mathcal{L}_D\underline{\varphi}_i\leq\left[\hat{V}(\kappa_i)+[\partial_{xx}\hat{m}]_+(\kappa_i)+2\epsilon\right]\underline{\varphi}_i
 &\text {in}\,\,((0,1)\setminus\{\kappa_i\pm\delta\})\times[0,T],\\
 \smallskip
 \partial_x\underline{\varphi}_i(0,t)=\partial_x\underline{\varphi}_i(1,t)=0  &\text {on}\,\,[0,T], \\
 \underline{\varphi}_i(x,0)=\underline{\varphi}_i(x,T) &\text{on}\,\,(0,1).
 \end{cases}
\end{equation*}
Then \eqref{upperestimate} is a direct consequence of  Proposition \ref{appendixprop}.

\medskip
\noindent {\bf Step 2. }
We  establish the lower bound of $\liminf_{D\rightarrow 0}\lambda(D)$.
It suffices to find a super-solution $\overline{\varphi}\in C([0,1]\times[0,T])$ satisfying \eqref{supersoltion}  with $\lambda_{\min}$ being replaced by the right hand side of \eqref{limitlambdaD} and  $\mathbb{X}$ will be determined later.
Recall the sets $\mathbf{A}$ and $\mathbf{B}$  defined in the statement of Theorem \ref{thm1.1}.
 The construction of $\overline\varphi$ can be given as follows; see Fig.\ref{figure1.4} for an illustrated example.  
\begin{figure}[http!!]
  \centering
\includegraphics[height=2.5in]{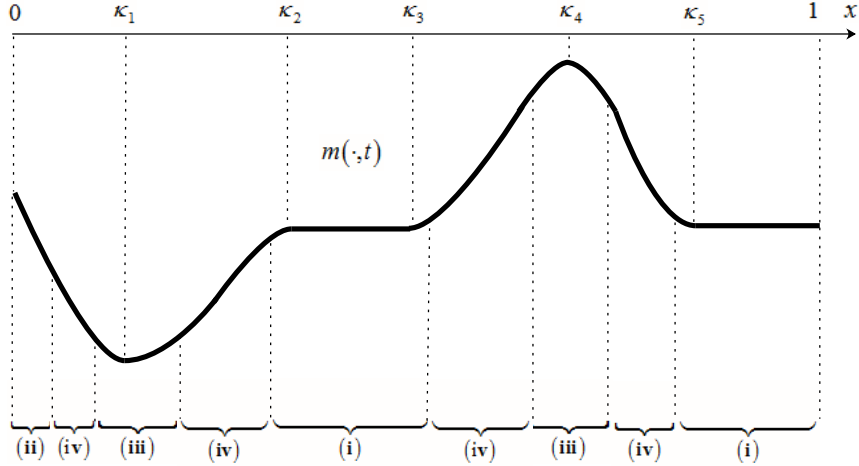}
  \caption{The black solid curve corresponds to an example of $m$ for fixed $t$.
 The super-solution $\overline\varphi$ is constructed respectively on different regions {\rm(i)-(v)}.
 }\label{figure1.4}
  \end{figure}

\noindent {\rm (i)} On $([\kappa_{i}-\delta,\kappa_{i+1}+\delta]\cap[0,1])\times[0,T]$ for $0\leq i\leq N$ and $i\in\mathbf{B}$ with  the small constant  $\delta>0$ to be determined later, we define $\overline\varphi$ as in the form of \eqref{liu000028} in Lemma \ref{sdlemdegen} with  $$\underline{\kappa}=\kappa_{i}-\delta, \quad \overline{\kappa}=\kappa_{i+1}+\delta,\quad\text{and}\quad c_1=c_2=\tfrac{1}{2}.$$

\noindent {\rm (ii)} On $([0,\frac{\kappa_1}{3}]\cup[\frac{2+\kappa_N}{3},1])\times[0,T]$, if $0\not\in\mathbf{B}$ or $N\not\in\mathbf{B}$, then such a super-solution $\overline{\varphi}$ can be constructed by adapting the same arguments as in the proof of Lemma \ref{sdlem5}; Otherwise, it has been constructed in {\rm (i)}.

\noindent {\rm (iii)} On $[\frac{\kappa_{i-1}+2\kappa_i}{3},\frac{2\kappa_{i}+\kappa_{i+1}}{3}]\times[0,T]$ for 
$i\in\mathbf{A}$ and $i-1\in\mathbf{A}$,  one constructs  $\overline{\varphi}$ by  Step 2 of  Proposition \ref{sdlem1} (with $\kappa=\kappa_i$) for the case $\partial_{xx}\hat{m}(\kappa_i)\leq 0$, and by Part II of  Proposition \ref{sdlem2} (with $\kappa=\kappa_i$)   for the  case $\partial_{xx}\hat{m}(\kappa_i)>0$.

\noindent {\rm (iv)} On the remaining region $\Omega\times [0,T]$, where
$$
\Omega=
\begin{cases}
\smallskip
(\frac{2\kappa_{i}+\kappa_{i+1}}{3}, \frac{\kappa_{i}+2\kappa_{i+1}}{3})&\text{ for } i\in\mathbf{A} \text{ and }i-1\in\mathbf{A},\\
\smallskip
(\kappa_{i-1}+\delta,\frac{2\kappa_{i}+\kappa_{i+1}}{3}) &\text{ for } i\in\mathbf{A} \text{ and }i-1\in\mathbf{B},\\
(\frac{\kappa_{i-1}+2\kappa_i}{3},\kappa_{i}-\delta) &\text{ for } i\in\mathbf{B},
\end{cases}
$$

we  construct $\overline\varphi$  by monotonically connecting the endpoints on $\partial \Omega$, such that
\begin{itemize}
  \item [{\rm(a)}] $\overline{\varphi}$ is continuous at  $\partial\Omega\times[0,T]$;
  \item [{\rm(b)}] $|\partial_x(\log\overline{\varphi})|>M_3$ for some large $M_3$;
  \item [{\rm(c)}] $\partial_x(\log\overline\varphi)(x^+,\cdot)<\partial_x(\log\overline\varphi)(x^-,\cdot)$ for $x\in \partial\Omega$.
\end{itemize}

Define $\mathbb{X}= \partial\Omega$.
By Lemmas \ref{sdlem5} and \ref{sdlemdegen},  explicit calculations as in Propositions \ref{sdlem1} and \ref{sdlem2} imply that we may choose $\delta$ smaller if necessary such that the super-solution $\overline\varphi$ defined above satisfies \eqref{supersoltion} with $\lambda_{\min}$ being replaced by the right hand side of \eqref{limitlambdaD}. Then  the lower bound of $\liminf_{D\rightarrow 0}\lambda(D)$ can be established by Proposition \ref{appendixprop}.
 The proof 
 is now complete.
\end{proof}

\smallskip
\section{Proof of Theorem \ref{sdthmP}}\label{S3}
In this section, we study the case when the ODE \eqref{eq:P} possesses finite periodic solutions and  establish Theorem \ref{sdthmP} with the help of Theorem \ref{thm1.1}. 

\begin{proof}[Proof of Theorem \rm\ref{sdthmP}]
 We first prove part {\rm(i)} of Theorem \ref{sdthmP}. Let $\{\kappa_i\}_{0\leq i\leq N+1}$ be any strictly increasing sequence  such that 
$$0=\kappa_0<\kappa_1<\ldots<\kappa_{N}<\kappa_{N+1}=1.$$
 Fix small $\delta$ such that $0<\delta<\min\limits_{0\leq i\leq N}(\kappa_{i+1}-\kappa_i)/3$ and
\begin{equation}\label{deltadef}
 \partial_{xx}m\left(x,t\right)\neq 0\quad\text{ for all }\,\, x\in[P_i(t)-\delta,P_i(t)+\delta],\, t\in[0,T], \,1\leq i\leq N.
\end{equation}

To ``straighten the periodic solution $P_i(t)$'', we first define a $C^{2,1}$-diffeomorphism  $\Psi:[0,1]\times[0,T]\rightarrow[0,1]$
such that  $\partial_y\Psi(y,t)\neq0$ 
and
\begin{equation}\label{eq:Psi}
\Psi(y,t)=
  \begin{cases}
  \smallskip
y-\kappa_i+P_i(t) & \text{ for  } y\in[\kappa_i-\delta,\kappa_i+\delta],\, t\in[0,T], \,1\leq i\leq N, \\
y  & \text{ for }y\in [0,\delta]\cup[1-\delta,1], \,t\in[0,T].
 \end{cases}
 \end{equation}
Define $\widetilde{V}(y,t)=V\left(\Psi(y,t),t\right)$.
By direct calculations,  $\lambda(D)$ is also the principal eigenvalue of
\begin{equation}\label{ldn_trans}
\begin{cases}
 \medskip
\partial_{t}\widetilde{\varphi}-D\frac{\partial_{yy}\widetilde{\varphi}}{(\partial_y\Psi)^{2}}- \Big[\partial_y\widetilde{m}
-D\frac{\partial_{yy}\Psi}{(\partial_y\Psi)^{3}}\Big]\partial_y\widetilde{\varphi}+\widetilde{V}(y,t)\widetilde{\varphi}
=\lambda(D)\widetilde{\varphi} \ \ &\text{ in } (0,1)\times [0,T],\\
\medskip
 \partial_{y}\widetilde{\varphi}(0,t)=\partial_{y}\widetilde{\varphi}(1,t)=0 \ \ & \text{ on } [0,T], \\
 \widetilde{\varphi}(y,0)=\widetilde{\varphi}(y,T) &\text{ on } (0,1),
\end{cases}
 \end{equation}
for  which the principal eigenfunction becomes $\widetilde{\varphi}(y,t)=\varphi\left(\Psi(y,t),t\right)$. Here   $\varphi$ denotes  the principal eigenfunction  of problem \eqref{SD_eq 1},  and $\widetilde{m}$ is given by
\begin{equation}\label{widehatm}
  \partial_y\widetilde{m}(y,t)=\frac{\partial_{x}m\left(\Psi(y,t),t\right)}{\partial_y\Psi}+\frac{\partial_{t}\Psi}{\partial_y\Psi}.
\end{equation}

 In what follows, we focus on problem \eqref{ldn_trans}, and divide the proof into several steps.

\medskip
\noindent {\bf Step 1.} We  show  that the ODE problem
 \begin{equation}\label{eq:hatP}
  \begin{cases}
  \smallskip
\dot{\widetilde{P}}(t)=-\partial_y\widetilde{m}(\widetilde{P}(t),t),\\
\widetilde{P}(t)=\widetilde{P}(t+T)
 \end{cases}
 \end{equation}
has only $N$ periodic solutions  $\widetilde{P}_i(t)\equiv\kappa_i$, and $\partial_{yy}\widetilde{m}\left(y,t\right)\neq 0$ for all $(y,t)\in[\kappa_i-\delta,\kappa_i+\delta]\times[0,T]$ and $i=1,\ldots,N$. 

First, we claim that  $\widetilde{P}_i(t)\equiv\kappa_i$ is a solution of \eqref{eq:hatP}. This is due to the following calculations:
\begin{equation*}
\begin{split}
  \partial_y\widetilde{m}(\kappa_i,t)&=\frac{\partial_{x}m(\Psi(\kappa_i,t),t)}{\partial_y\Psi(\kappa_i,t)}+\frac{\partial_{t}\Psi(\kappa_i,t)}{\partial_y\Psi(\kappa_i,t)}\\
&=\partial_{x}m(P_i(t),t)+\dot{P}_i(t)=0,
\end{split}
\end{equation*}
where the first equality follows from  \eqref{widehatm},
and the second equality is due to \eqref{eq:Psi}.

Suppose on the contrary that there exists a periodic solution $\widetilde{P}(t)$ such that $\widetilde{P}(t)\not\equiv\kappa_i$ for any $1\leq i\leq N$. Then by \eqref{eq:Psi} and \eqref{widehatm}, one can verify  that $\Psi(\widetilde{P}(t),t)\not\equiv P_i(t)$ is a periodic solution to \eqref{eq:P} by the following calculations:
\begin{equation*}
\begin{split}
  \dot{\Psi}(\widetilde{P}(t),t)=
\dot{\widetilde{P}}(t)\partial_y\Psi+\partial_t\Psi&=-\partial_y\widetilde{m}(\widetilde{P}(t),t)\partial_y\Psi+\partial_t\Psi\\
&=-\partial_{x}m(\Psi(\widetilde{P}(t),t),t)-\partial_{t}\Psi+\partial_{t}\Psi\\
&=-\partial_{x}m(\Psi(\widetilde{P}(t),t),t),
\end{split}
\end{equation*}
which is a contradiction. Therefore, \eqref{eq:hatP} has only $N$ periodic solutions  $\widetilde{P}_i(t)\equiv\kappa_i$ ($i=1,\ldots,N$). Furthermore,  from \eqref{deltadef} and \eqref{eq:Psi}, it is easily seen that $\partial_{yy}\widetilde{m}\left(y,t\right)\neq 0$ on $[\kappa_i-\delta,\kappa_i+\delta]\times[0,T]$,  
which completes Step 1.

%
In the sequel, we aim to find a proper $C^{2,1}$-transformation $\Phi:[0,1]\times\mathbb{R}\rightarrow[0,1]$ such that $\partial_z \Phi>0$, and
if for some $\overline{m}\in C^{2,1}([0,1]\times[0,T])$ satisfying
\begin{equation}\label{defoverlinem}
\partial_z\overline{m}(z,r)= \frac{\partial_y\widetilde{m}\left(\Phi(z,r),r\right)}{\partial_z\Phi}+\frac{\partial_r\Phi}{\partial_z\Phi},
\end{equation}
 then $\partial_z\overline{m}> 0$ or $\partial_z\overline{m}< 0$ holds on $(\kappa_i,\kappa_{i+1})\times[0,T]$
for each $0\leq i\leq N$.
Then we may apply Theorem \ref{thm1.1} to complete the proof.

Fix any $0\leq i\leq N$. We assume without loss of generality that $\partial_{yy}\widetilde{m}\left(\kappa_i,t\right)<0$, so that 
$\partial_{y}\widetilde{m}\left(\kappa_i+\delta/2,t\right)<0$. For any $s\in\mathbb{R}$, denote by $q_s(t)$ the unique solution of 
\begin{equation}\label{eq:q}
  \begin{cases}
  \smallskip
\dot{q}(t)=-\partial_y\widetilde{m}(q(t),t+s),\\
 q(0)=\kappa_i+\delta/2,
 \end{cases}
 \end{equation}
where $\widetilde{m}$ is given by \eqref{widehatm}.
Obviously, $q_s(t)=q_{s+T}(t)$ for all $s,t\in \mathbb{R}$.
We define
\begin{equation}\label{defQt}
 Q(t):=\left\{(q_{r-t}(t),r):r\in\mathbb{R}\right\}, 
\end{equation}
which is  a continuous curve and is referred  as the isochron of \eqref{eq:q}.

\medskip
\noindent {\bf Step 2.} Fix any $0<t_1<t_2$. We show  that $Q(t_1)\prec Q(t_2)$ (see Fig.\ref{figure2}) in the sense that
\begin{equation}\label{monotonicty1}
  q_{r-t_1}(t_1)<q_{r-t_2}(t_2)\quad\text{ for any } \,\,r\in \mathbb{R}.
\end{equation}

We argue by contradiction by assuming $Q(t_1)\cap Q(t_2)\neq\emptyset$ or $Q(t_2)\prec Q(t_1)$.

{\rm (i)} If $Q(t_1)\cap Q(t_2)\neq\emptyset$, then by definition \eqref{defQt}, there exists some $r_0\in\mathbb{R}$ such that
\begin{equation}\label{relation1}
  q_{r_0-t_1}(t_1)=q_{r_0-t_2}(t_2).
\end{equation}
Then we define
$$\overline{q}(t):=q_{r_0-t_1}(t-r_0+t_1)\quad\text{and}\quad\underline{q}(t):=q_{r_0-t_2}(t-r_0+t_2),$$
both of which satisfy $\dot{q}(t)=-\partial_y\widetilde{m}\left(q(t),t\right)$, and
\begin{equation}\label{initialcondition}
  \overline{q}(r_0-t_1)=\underline{q}(r_0-t_2)=\kappa_i+\delta/2\quad\text{ and }\quad \overline{q}(r_0)=\underline{q}(r_0),
\end{equation}
where $\overline{q}(r_0)=\underline{q}(r_0)$ follows from \eqref{relation1}. In view of $t_1<t_2$, we have $r_0-t_1>r_0-t_2$. Thanks to the uniqueness of solutions to $\dot{q}(t)=-\partial_y\widetilde{m}\left(q(t),t\right)$, we conclude from \eqref{initialcondition} that
 $ \overline{q}(t)=\underline{q}(t)$ for any $t\in [r_0-t_1,r_0]$,
and particularly, $\underline{q}(r_0-t_1)=\overline{q}(r_0-t_1)=\kappa_i+\delta/2=\underline{q}(r_0-t_2)$, i.e.
$q_{r_0-t_2}(t_2-t_1)=\kappa_i+\delta/2=q_{r_0-t_2}(0),$
 which contradicts $\partial_{y}\widetilde{m}\left(\kappa_i+\delta/2,t\right)<0$.

{\rm (ii)} If $Q(t_2)\prec Q(t_1)$, then given any $(q_{r_1-t_1}(t_1),r_1)\in Q(t_1)$, there is some $t_0\in(0,t_1)$ such that $(q_{r_1-t_1}(t_0),r_2)\in Q(t_2)$, where $r_2=r_1-(t_1-t_0)$.
By definition \eqref{defQt},  we also have $(q_{r_2-t_2}(t_2),r_2)\in Q(t_2)$, so that
$q_{r_1-t_1}(t_0)=q_{r_2-t_2}(t_2)$.
This, together with $r_2-t_0=r_1-t_1$, leads to $q_{r_2-t_0}(t_0)=q_{r_2-t_2}(t_2)$, whence
$(q_{r_2-t_2}(t_2),r_2)\in Q(t_0)\cap Q(t_2)$, i.e. $Q(t_0)\cap Q(t_2)\neq\emptyset$. Since $t_0<t_2$, we can apply  {\rm (i)} to reach a contradiction.

\medskip
\noindent {\bf Step 3.} We show
$$\lim_{t\to\infty}Q(t)=\{(\kappa_{i+1},r):r\in\mathbb{R}\},$$
in the sense that for any $r\in\mathbb{R}$, $q_{r-t}(t)\to \kappa_{i+1}$ as $t\to \infty$.

 By $\mathbb{M}$ we denote the set of all continuous curves in $[\kappa_i+\delta/2,\kappa_{i+1}]\times[0,T]$.  By Step 2, there is some curve $Q_\infty:=\{(q_\infty(s),s):s\in\mathbb{R}\}\in \mathbb{M}$ such that $Q(t)\to Q_\infty$ as $t\to\infty$.  It suffices to show $q_\infty\equiv\kappa_{i+1}$.
To this end, we  claim that $q_\infty$ is a periodic solution of \eqref{eq:hatP}, and then $q_\infty\equiv\kappa_{i+1}$  is a direct consequence of Step 1.

Indeed, the periodicity of $q_\infty$ is due to the fact that $q_s(t)=q_{s+T}(t)$ for all $s,t\in \mathbb{R}$.  We show that  $q_\infty$ is a solution to \eqref{eq:hatP}. Suppose not, then for given $s_0\in\mathbb{R}$, there exists  some $t_0>s_0$  such that the unique solution $p_{s_0}(t)$ of
\begin{equation*}
  \begin{cases}
  \smallskip
\dot{p}(t)=-\partial_y\widetilde{m}\left(p(t),t+s_0\right),\\
 p(0)=q_\infty(s_0),
 \end{cases}
 \end{equation*}
satisfies
$p_{s_0}(t_0-s_0)\neq q_\infty(t_0)$. Let $t_*=t_0-s_0$.
For any $\Sigma_q:=\{(q(s),s):s\in\mathbb{R}\}\in \mathbb{M}$, we denote by $p_s$ the unique solution of
\begin{equation*}
  \begin{cases}
  \smallskip
\dot{p}(t)=-\partial_y\widetilde{m}\left(p(t),t+s\right),\\
 p(0)=q(s),
 \end{cases}
 \end{equation*}
and define a continuous operator $F:\mathbb{M}\to \mathbb{M}$ by
$$F(\Sigma_q):=\left\{(p_s(t_*),t_*+s):s\in\mathbb{R}\right\}=\left\{(p_{r-t_*}(t_*),r):r\in\mathbb{R}\right\}.$$

It is straightforward to verify that $F(Q(t))=Q(t+t_*)$, and thus
$$F(Q_\infty)=Q_\infty,$$
from which  we deduce  in particular that $p_{t_0-t_*}(t_*)=q_\infty(t_0)$, that is $p_{s_0}(t_0-s_0)=q_\infty(t_0)$, a contradiction. Therefore, $q_\infty$ is a periodic solution of \eqref{eq:hatP}. Step 3 is thereby completed.

\medskip
\noindent {\bf Step 4.} We define the transformation $\Phi$ satisfying $\partial_z\Phi>0$, and for $\overline{m}$ given by \eqref{defoverlinem},
we show that  $\partial_z\overline{m}> 0$ or $\partial_z\overline{m}< 0$ holds in $(\kappa_i,\kappa_{i+1})\times[0,T]$ for each $0\leq i\leq N$.

For any $0\leq i \leq N$, we  define
$\Phi_i:[\kappa_i+\delta/2,\kappa_{i+1}-\delta/2]\times\mathbb{R}\rightarrow[\kappa_i,\kappa_{i+1}]$
such that for any $(z,r)\in [\kappa_i+\delta,\kappa_{i+1}-\delta]\times\mathbb{R}$,
\begin{equation}\label{Phii}
  \Phi_i(z,r)=q_{r-\tau_i(z)}(\tau_i(z)),
\end{equation}
where $q_{r-\tau_i(z)}$ is the solution of \eqref{eq:q} with $s=r-\tau_i(z)$ and $\tau_i(z)$ is determined by
\begin{equation}\label{tkappa}
 q_{-\tau_i(z)}(\tau_i(z))=z.
\end{equation}
Obviously, $\{(\Phi_i(z,r),r):r\in\mathbb{R}\}=Q(\tau_i(z))$. 
It is easily seen that $z\to\tau_i(z)$ is a bijection (where the surjection follows from Step 3), is of class $C^2$ and is increasing (by Step 2), so that
$\Phi_i\in C^{2,1}([\kappa_i+\delta,\kappa_{i+1}-\delta]\times\mathbb{R})$ and  $\partial_z\Phi_i\geq0$  by \eqref{monotonicty1}.


 We claim that for $(z,r)\in(\kappa_i+\delta/2,\kappa_{i+1}-\delta/2)\times\mathbb{R}$,
\begin{equation}\label{liu000014}
    \tau'_i(z)>0 \quad\text{and}\quad \partial_y\widetilde{m}\left(\Phi_i(z,r),r\right)+\partial_r\Phi_i=-\frac{\partial_z\Phi_i}{\tau'_i(z)}.
\end{equation}

For the sake of clarification, write $q_s(t)=q(t;s)$, where  $q_s$ is defined by \eqref{eq:q}. Differentiating both sides of \eqref{tkappa} by $z$, we derive that
$$\Big[\partial_t q_{-\tau_i(z)}(\tau_i(z))-\partial_s q_{-\tau_i(z)}(\tau_i(z))\Big]\tau'_i(z)=1,$$
which implies $\tau'_i(z)\neq 0$, and thus $\tau'_i(z)>0$  since $\tau_i(z)$ is increasing. Similarly, by \eqref{Phii}, we deduce that
$\partial_r\Phi_i(z,r)=\partial_sq_{r-\tau_i(z)}(\tau_i(z))$, and thus
\begin{equation}\label{liu000015}
    \begin{split}
    \partial_z\Phi_i(z,r)&=\Big[\partial_tq_{r-\tau_i(z)}(\tau_i(z))-\partial_sq_{r-\tau_i(z)}(\tau_i(z))\Big]\tau'_i(z)\\
    &=\Big[\partial_tq_{r-\tau_i(z)}(\tau_i(z))-\partial_r\Phi_i(z,r)\Big]\tau'_i(z).
\end{split}
\end{equation}
By the definition of $q_{r-\tau_i(z)}(\tau_i(z))$ in \eqref{eq:q} with $s=r-\tau_i(z)$ and $t=\tau_i(z)$, we note that
$$\partial_t q_{r-\tau_i(z)}(\tau_i(z))=-\partial_y\widetilde{m}(\Phi_i(z,r),r),$$
which, together with \eqref{liu000015}, implies \eqref{liu000014}.

  We then claim that
\begin{equation}\label{liu000020}
     \partial_z\Phi_i(z,r)>0 \quad\text{for any}\quad (z,r)\in(\kappa_i+\delta/2,\kappa_{i+1}-\delta/2)\times\mathbb{R}.
\end{equation}

To this end,  denote by  $\tilde{p}(t;s)$  the unique solution of the  problem
 \begin{equation*}
  \begin{cases}
  \smallskip
\dot{\tilde{p}}(t)=-\partial_y\widetilde{m}\left(\tilde{p}(t),t\right),\\
 \tilde{p}(s)=\kappa_i+\delta/2,
 \end{cases}
 \end{equation*}
 whence by \eqref{eq:q}, we observe that  $q_s(t)=\tilde{p}(t+s;s)$.
For any $\tau\in\mathbb{R}$, we have
$$\tilde{p}(\tau)=-\int_s^\tau\partial_y\widetilde{m}\left(\tilde{p}(t),t\right)\mathrm{d}t+\kappa_i+\delta/2,$$
so that 
$$\partial_s\tilde{p}(\tau)=\partial_y\widetilde{m}\left(\kappa_i+\delta/2,s\right)-\int_s^\tau\partial_{yy}\widetilde{m}\left(\tilde{p}(t),t\right)\partial_s\tilde{p}(t)\mathrm{d}t,$$
and thus $\partial_s\tilde{p}(s)=\partial_y\widetilde{m}\left(\kappa_i+\delta/2,s\right)<0$. We further calculate that
$$\partial_\tau\left(\partial_s\tilde{p}(\tau)\right)=-\partial_{yy}\widetilde{m}\left(\tilde{p}(\tau),\tau\right)\partial_s\tilde{p}(\tau),$$
which implies immediately that for any $r\in \mathbb{R}$,
\begin{equation}\label{liu000022}
    \partial_s\tilde{p}(r) =\partial_s\tilde{p}(s)\exp\left[-\int_{s}^{r}\partial_{yy}\widetilde{m}\left(\tilde{p}(\tau),\tau\right)\mathrm{d}\tau\right]<0.
\end{equation}
By \eqref{Phii} and the fact that  $q_s(t)=\tilde{p}(t+s;s)$, we can see   $\Phi_i(z,r)=\tilde{p}(r;r-\tau_i(z))$,
so that
$$\partial_z\Phi_i(z,r)=-\partial_s\tilde{p}(r)\cdot\tau'_i(z)>0$$
by noting that  $\tau'_i(z)>0$  in \eqref{liu000014} and $\partial_s\tilde{p}(r)<0$ in \eqref{liu000022}.

Then we define  a $C^{2,1}$-transformation $\Phi:[0,1]\times\mathbb{R}\rightarrow[0,1]$ such that
 $\partial_z\Phi>0$ and for any $0\leq i\leq N$,
\begin{equation}\label{defphi}
  \Phi(z,r):=\begin{cases}
  \smallskip
\Phi_i(z,r) &\text{ on } [\kappa_i+\delta_1,\kappa_{i+1}-\delta_1]\times\mathbb{R}, \\
 z  &\text{ on } \left([\kappa_i,\kappa_{i}+\delta/2]\cup[\kappa_{i+1}-\delta/2,\kappa_{i+1}]\right)\times\mathbb{R},\\
 \end{cases}
 \end{equation}
where $\delta_1\in(\delta/2,\delta]$ is  chosen to be   close to $\delta/2$ such that
\begin{equation}\label{delta1}
  \partial_y\widetilde{m}+\partial_r\Phi< 0\,\,\,\,\text{ on }\,\,\left([\kappa_i,\kappa_{i}+\delta_1]\cup[\kappa_{i+1}-\delta_1,\kappa_{i+1}]\right)\times\mathbb{R}. 
\end{equation}
This is possible since  by \eqref{defphi} and Step 1, it follows that
$$\partial_y\widetilde{m}+\partial_r\Phi=\partial_y\widetilde{m}<0 \,\,\text{ on } \left([\kappa_i,\kappa_{i}+\delta/2]\cup[\kappa_{i+1}-\delta/2,\kappa_{i+1}]\right)\times\mathbb{R}.$$

  Let $\overline{m}$ satisfy \eqref{defoverlinem} with $\Phi$ defined by \eqref{defphi}.
For any $z\in [\kappa_i,\kappa_{i}+\delta_1]\cup[\kappa_{i+1}-\delta_1,\kappa_{i+1}]$,  it follows from  \eqref{defoverlinem}, \eqref{defphi}, and \eqref{delta1}  that $\partial_z\overline{m}(z,r)< 0$; For $z\in[\kappa_i+\delta_1,\kappa_{i+1}-\delta_1]$, by \eqref{defphi} we have $\Phi(z,r)=\Phi_i(z,r)$, whence comparing \eqref{defoverlinem} with \eqref{liu000014} gives $\partial_z\overline{m}(z,r)=-\frac{1}{\tau'_i(z)}<0$.
This completes  Step 4.

\medskip
\noindent {\bf Step 5.}  We apply Theorem \ref{thm1.1} to complete the proof. Let the $C^{2,1}$-transformation  $\Phi$ be defined by \eqref{defphi}  in Step 4.
Denote
$$\overline{V}(z,r):=\widetilde{V}\left(\Phi(z,r),r\right)\quad\text{ and }\quad \overline{\varphi}(z,r):=\widetilde{\varphi}\left(\Phi(z,r),r\right),$$
where $\widetilde{V}$ and $\widetilde{\varphi}$ are defined in \eqref{ldn_trans}.
 Using the definition of $\overline{m}$ in \eqref{defoverlinem}, direct calculation enables us to transform \eqref{ldn_trans}  into the following equation:
\begin{equation}\label{ldn_trans2}
\begin{cases}
 \begin{array}{l}
 \smallskip
\partial_{r}\overline{\varphi}-\frac{D\partial_{zz}\overline{\varphi}}{(\partial_y\Psi)^{2}(\partial_z\Phi)^{2}} - \Big[\partial_{z}\overline{m}+
D\eta_3\Big]\partial_z\overline{\varphi}+\overline{V}\overline{\varphi}
=\lambda(D)\overline{\varphi}\end{array}
\ \ &\text{ in } (0,1)\times [0,T],\\
\medskip
 \partial_{z}\overline{\varphi}(0,r)=\partial_{z}\overline{\varphi}(1,r)=0 \ \ & \text{ on } [0,T], \\
 \overline{\varphi}(z,0)=\overline{\varphi}(z,T) &\text{ on } (0,1),
\end{cases}
 \end{equation}
where $\eta_3$ is given by
$$
\eta_3(z,r):=\frac{\partial_{yy}\Psi}{(\partial_y\Psi)^{3}\partial_z\Phi}+\frac{\partial_{zz}\Phi}{(\partial_z\Phi)^{3}(\partial_y\Psi)^2}.
$$

For each  $0\leq i\leq N$, by Step 4,  $\partial_z\overline{m}> 0$ or $\partial_z\overline{m}< 0$ holds for all $z\in(\kappa_i,\kappa_{i+1})$; by the definitions of $\Psi$ and $\Phi$ in \eqref{eq:Psi} and \eqref{defphi},  we find that for any  $z\in[\kappa_i,\kappa_{i}+\delta/2]\cup[\kappa_{i+1}-\delta/2,\kappa_{i+1}]$,  $\partial_{yy}\Psi=\partial_{zz}\Phi=0$, so that $\eta_3(z,r)=0$.
 Therefore, we conclude that for any small $\vartheta>0$, there exists some $\epsilon_0=\epsilon_0(\vartheta)>0$, independent of small $D$, such that
\begin{equation}\label{liu000023}
\partial_{z}\overline{m}+D\eta_3\geq\epsilon_0 \quad\text{or}\quad \partial_{z}\overline{m}+D\eta_3\leq-\epsilon_0\quad \text{on} \,\,\,[\kappa_i+\vartheta,\kappa_{i+1}-\vartheta]\times [0,T].
\end{equation}

Moreover, from \eqref{defoverlinem} and \eqref{defphi}, we observe  that for any $1\leq i\leq N$,
$$\partial_{z}\overline{m}(\kappa_i,r)=\partial_y\widetilde{m}\left(\Phi(\kappa_i,r),r\right)=\partial_y\widetilde{m}\left(\kappa_i,r\right)=0,$$
which implies that $\partial_{z}\overline{m}(\kappa_i,r)+D\eta_3(\kappa_i,r)=0$ since $\eta_3(\kappa_i,r)=0$. Together with \eqref{liu000023}, noting that $\partial_y\Psi=\partial_z\Phi=1$ on  $[\kappa_i,\kappa_{i}+\tilde\vartheta]\cup[\kappa_{i+1}-\tilde\vartheta,\kappa_{i+1}]$ with some $0<\tilde\vartheta\ll 1$ for any $0\leq i \leq N$, we can follow directly
the same proof of Theorem \ref{thm1.1} with $\mathbf{B}=\emptyset$ to \eqref{ldn_trans2} and
deduce that
\begin{equation}\label{liu000017}
\lim_{D\rightarrow 0}\lambda(D)=\min_{0\leq i\leq N+1}\left\{\hat{\overline{V}}(\kappa_i)+[\partial_{xx}\hat{m}]_+(\kappa_i)\right\}.
\end{equation}
Noting that $\hat{\overline{V}}(\kappa_i)=\frac{1}{T}\int_0^T V\left(P_i(s),s\right)\mathrm{d}s$ and
$$\partial_{zz}\overline{m}(\kappa_i,r)=\partial_{yy}\widetilde{m}\left(\kappa_i,r\right)=\partial_{xx}m\left(P_i(r),r\right),$$
part {\rm(i)} of Theorem \ref{sdthmP} follows from \eqref{liu000017}.

Finally, part {\rm(ii)} of Theorem \ref{sdthmP} can be established by  Steps 2-5 with $N=0$. The proof is now complete.
\end{proof}

\medskip

\section{Proof of Theorem \ref{sdthm4}}\label{S4}
This section is devoted to  the case 
$m(x,t)=\alpha b(t)x$ and the proof of Theorem \ref{sdthm4}. 
We start with the existence and uniqueness of $\tilde{P}_\alpha$ defined in Theorem \ref{sdthm4}.
\begin{lemma}\label{existenceandunique}
Let $F$ be defined by \eqref{definition_F}
and $\overline{P}$,$\underline{P}$ be given in Theorem {\rm\ref{sdthm4}}. 
If $\alpha\geq\frac{1}{\overline{P}-\underline{P}}$, then 
 \begin{equation}\label{eq:tildeP}
  \begin{cases}
  \smallskip
\dot{P}(t)=-\alpha F\left(P(t),t\right),\\
P(t)=P(t+T)
 \end{cases}
 \end{equation}
has a unique $T$-periodic solution in
$W^{1,\infty}(\mathbb{R})$.
\end{lemma}
\begin{proof}
Recalling the definition of $F$ in \eqref{definition_F}, 
we observe that  $F(0,t)=\min\{b(t), 0\}\leq0$  and $F(1,t)=\max\{b(t), 0\}\geq0$, so that
$P_*\equiv0$ and $P^*\equiv1$ are a pair of sub- and super-solutions to \eqref{eq:tildeP}. Hence,
as $F$ is bounded,
there exists at least one $T$-periodic solution in  $W^{1,\infty}(\mathbb{R})$.

 For the uniqueness, given any two $T$-periodic solutions $\tilde{P}$ and $\tilde{P}_\alpha$ of  \eqref{eq:tildeP}, we 
 show $\tilde{P}=\tilde{P}_\alpha$. Suppose not, 
without loss of generality
 we may assume $\tilde{P}(0)<\tilde{P}_\alpha(0)$. We 
 consider two cases:
\smallskip

  \noindent {\bf (i)} If $\tilde{P}(t_1)=\tilde{P}_\alpha(t_1)$ for some $t_1\in(0,T)$,  in view of  $\tilde{P}(T)=\tilde{P}(0)<\tilde{P}_\alpha(0)=\tilde{P}_\alpha(T)$, by continuity there is some $\tilde t_1\in [t_1,T)$ such that $\tilde{P}(\tilde t_1)=\tilde{P}_\alpha(\tilde t_1)$ and $\tilde{P}(t)<\tilde{P}_\alpha(t)$ for any $t\in (\tilde t_1, T]$. Then
  by the  definition of $F$,
  it can be verified that for any $t\in  [\tilde t_1, T]$,
  \begin{equation}\label{liu0010}
   [\tilde{P}_\alpha(t)-\tilde{P}(t)]- [\tilde{P}_\alpha(\tilde t_1)-\tilde{P}(\tilde t_1)] =\alpha \int_{\tilde t_1}^t\left[F(\tilde{P}(s),s)-F(\tilde{P}_\alpha(s),s)\right]\, ds\leq 0,
  \end{equation}
 which implies $\tilde{P}_\alpha(T)-\tilde{P}(T)\leq \tilde{P}_\alpha(\tilde t_1)-\tilde{P}(\tilde t_1)=0$, a contradiction.

\smallskip
  \noindent {\bf (ii)} If  $\tilde{P}(t)<\tilde{P}_\alpha(t)$ for all $t\in[0,T]$, then \eqref{liu0010} holds for all $t\in[0,T]$ and $\tilde{t}_1=0$.
In view of $\tilde{P}_\alpha(T)-\tilde{P}(T)=\tilde{P}_\alpha(0)-\tilde{P}(0)$, we deduce that $$\tilde{P}_\alpha(t)-\tilde{P}(t)\equiv\tilde{P}_\alpha(0)-\tilde{P}(0)\quad\text{for all}\quad t\in[0,T].$$
In such a case, again by the definition of $F$, we infer that
\begin{equation*}
 \tilde{P}_+:=\tilde{P}+(\tilde{P}_\alpha(0)-\tilde{P}(0))/2\in(0,1)
\end{equation*}
 defines a $T$-periodic solution of \eqref{eq:tildeP}, and thus $ \dot{\tilde{P}}_+=-\alpha b(t)$, where $\tilde{P}_+\in(0,1)$ is due to
 $0\leq\tilde{P}<\tilde{P}_+<\tilde{P}_\alpha\leq 1$.  By recalling $P(t)=-\int_{0}^{t}b(s)\mathrm{d}s$, this implies that
 $\tilde{P}_+=\alpha P(t)+c\in(0,1)$
 for some constant $c\in\mathbb{R}$, so that
 $$1>\max_{[0,T]}\tilde{P}_+-\min_{[0,T]} \tilde{P}_+=\alpha(\overline{P}-\underline{P}),$$
  which contradicts $\alpha\geq\frac{1}{\overline{P}-\underline{P}}$. Lemma \ref{existenceandunique} thus follows.
  \end{proof}
We are now ready to prove Theorem \ref{sdthm4}.

\begin{proof}[Proof of Theorem  {\rm\ref{sdthm4}}]
The proof is 
divided into 
three steps. 

\medskip

{\bf Step 1.} Assume $\hat{b}\neq0$ and show part {\rm(i)}  of Theorem  \ref{sdthm4}.
Let $\Psi_1:[0,1]\times[0,T]\rightarrow \mathbb{R}$ denote a $T$-periodic diffeomorphism
given by
$$\Psi_1(y,t)=\alpha\left[\hat{b}t-\int_{0}^{t}b(s)\,\mathrm{d}s\right]+y.$$
Under the transformation $x=\Psi_1(y,t)$, as in \eqref{ldn_trans}, direct calculation from  \eqref{SD_eq 4} yields that $\lambda(D)$ defines the principal eigenvalue of  the problem
\begin{equation*}
\begin{cases}
 \smallskip
\partial_{t}\varphi-D\partial_{yy}\varphi- \alpha\hat{b}\cdot\partial_y\varphi+V_1\varphi
=\lambda(D)\varphi, \ \ &y\in (-\Psi_1(0,t),1-\Psi_1(0,t)), t\in [0,T],\\
\smallskip
 \partial_{y}\varphi(-\Psi_1(0,t),t)=\partial_{y}\varphi(1-\Psi_1(0,t),t)=0, \ \ & t\in[0,T], \\
 \varphi(y,0)=\varphi(y,T), &y\in [-\Psi_1(0, 0),1-\Psi_1(0, 0)],
\end{cases}
 \end{equation*}
where $V_1(y,t)=V\left(\Psi_1(y,t),t\right)$.  Then we can conclude that part {\rm(i)}  of Theorem  \ref{sdthm4} is a direct consequence of Theorem \ref{sdthmP}. Indeed, if $\hat{b}>0$ for example, then ODE \eqref{eq:P} with $\partial_xm=\alpha \hat{b}>0$ has no periodic solutions, so that by part {\rm(ii)} of Theorem \ref{sdthmP} we deduce that
$$\lim_{D\to 0}\lambda(D)=\frac{1}{T}\int_{0}^{T}V_1(1-\Psi_1(0,s),s)\mathrm{d}s=\hat{V}(1).$$
The same argument can be adapted to the case $\hat{b}<0$, which completes Step 1.

\medskip
{\bf Step 2.} Assume $\hat{b}=0$ and $0<\alpha\leq \frac{1}{\overline{P}-\underline{P}}$. We prove the first part of {\rm(ii)} in Theorem  \ref{sdthm4}.
Recall $P(t)=-\int_{0}^{t}b(s)\mathrm{d}s$ defined in Theorem  \ref{sdthm4}. Taking the transformation $x=y+\alpha P(t)$ in \eqref{SD_eq 4}, we derive that $\lambda(D)$ is also the principal eigenvalue of the problem
\begin{equation*}
\begin{cases}
 \smallskip
\partial_{t}\varphi-D\partial_{yy}\varphi+V_2\varphi
=\lambda(D)\varphi, \ \ &y\in (-\alpha P(t),1-\alpha P(t)),\, t\in [0,T],\\
\smallskip
 \partial_{y}\varphi(-\alpha P(t),t)=\partial_{y}\varphi(1-\alpha P(t),t)=0, \ \ & t\in [0,T], \\
\varphi(y,0)=\varphi(y,T), &y\in (-\alpha P(0),1-\alpha P(0)),
\end{cases}
 \end{equation*}
 where $V_2(y,t)=V\left(\alpha P(t)+y,t\right)$. Under the transformation $x=y+\alpha P(t)$,  all periodic solutions of \eqref{eq:P} are constants in the interval $ [-\alpha\underline{P},1-\alpha\overline{P}]$. This includes the special case $\alpha=\frac{1}{\overline{P}-\underline{P}}$, for which the interval reduces to a single point.  It is desired to show that
 \begin{equation*}
   \lim\limits_{D\rightarrow0}\lambda(D)=\min_{y\in[-\alpha\underline{P}, \,1-\alpha\overline{P}]}\hat{V}_2(y).
 \end{equation*}

First, the upper bound 
$\limsup_{D\rightarrow 0}\lambda(D)\leq \hat{V}_2(y)$, for any $y\in[-\alpha\underline{P},1-\alpha\overline{P}]$,
can be established by the same arguments as in Step 1 of Lemma \ref{sdlemdegen} by constructing the sub-solution locally. We thus omit the details here.

It remains to show the lower bound of $\liminf_{D\rightarrow 0}\lambda(D)$.
For any $\epsilon>0$,  we define  $T$-periodic function $V_{2\epsilon}\in C^{2,1}(\mathbb{R}\times[0,T])$ satisfying $\|V_{2\epsilon}-V_2\|_{L^\infty}\leq \epsilon$, and choose small $\delta>0$ such  that
\begin{equation}\label{liu000027}
\tilde{\lambda}_{\rm min}:=\hspace{-0.2cm}\min_{y\in[-\alpha\underline{P}-2\delta,1-\alpha\overline{P}+2\delta]}\hat{V}_{2\epsilon}(y)\geq \min_{y\in[-\alpha\underline{P},1-\alpha\overline{P}]}\hat{V}_2(y)- 2\epsilon.
\end{equation}

We define $\overline{\phi}\in C^{2,1}([-\alpha\underline{P}-2\delta,\,1-\alpha\overline{P}+2\delta]\times[0,T])$ by
\begin{equation}\label{liu000025}
\overline{\phi}(y,t):=\mbox{exp} \left[-\int_0^t V_{2\epsilon}(y,s)\mathrm{d}s+t\hat{V}_{2\epsilon}(y)\right]\beta_{\epsilon} (y),
\end{equation}
where  $\beta_{\epsilon}\in C^2([-\alpha\underline{P}-2\delta,\,1-\alpha\overline{P}+2\delta])$ is a positive function chosen such that
\begin{equation}\label{liu000026}
\partial_y\overline{\phi}<0 \,\,\text{ on }[-\alpha\underline{P}-2\delta,-\alpha\underline{P}]\times[0,T]\quad\text{and}\quad\partial_y\overline{\phi}>0\,\,\text{ on }[1-\alpha\overline{P},1-\alpha\overline{P}+2\delta]\times[0,T].
\end{equation}

Next, we aim to find a super-solution  $\overline \varphi\in C([0,1]\times[0,T])$ which satisfies 
 \begin{equation}\label{supersoltionprop4.1}
 \begin{cases}
  \smallskip
 \partial_{t}\overline{\varphi}-D\partial_{yy}\overline{\varphi}+V_2\overline{\varphi}\geq \left[\tilde{\lambda}_{\rm min}-3\epsilon\right]\overline{\varphi},  &y\in (-\alpha P(t),1-\alpha P(t))\backslash\mathbb{X},\, t\in [0,T],\\
 \smallskip
 \partial_{y}\overline{\varphi}(-\alpha P(t),t)\leq0 \le \partial_{y}\overline{\varphi}(1-\alpha P(t),t), & t\in [0,T], \\
\overline{\varphi}(y,0)=\overline{\varphi}(y,T), &y\in (-\alpha P(0), 1-\alpha P(0)),
 \end{cases}
\end{equation}
 where $\mathbb{X}=\{-\alpha\underline{P}-2\delta,\,1-\alpha\overline{P}+2\delta\}$.
Then it  follows from  Proposition \ref{appendixprop} and \eqref{liu000027} that
  \begin{equation*}
   \liminf\limits_{D\rightarrow0}\lambda(D)\geq\min_{y\in[-\alpha\underline{P}, \, 1-\alpha\overline{P}]}
   \hat{V}_2(y);
  \end{equation*}
see also Remark  \ref{appendixrem}. 

We only construct $\overline{\varphi}$ for $y\in(-\alpha P(t),1-\alpha\overline{P}+\delta)$ and $t\in[0,T]$. The constructions of the remaining regions are similar. To this end,  by the definition of $\underline{P}$, there exist $t_3>t_2$ such that
$$[t_2,t_3]\subset\{t\in[0,T]:-\alpha P(t)>-\alpha\underline{P}-\delta\}.$$
We then choose 
$\eta_4\in C^{2,1}((-\infty,1-\alpha\overline{P}+\delta]\times[0,T])$
to be a positive $T$-periodic function, and satisfy that $\partial_y\eta_4\leq0$ and
 \begin{equation}\label{eta4}
 \begin{cases}
 \medskip
 \eta_4\equiv 1 &\text{on }[-\alpha\underline{P}-\delta,\,1-\alpha\overline{P}+\delta]\times [0,T],\\
 \medskip
 \partial_t (\log\eta_4)>0 &\text{on }[-\alpha\underline{P}-2\delta,-\alpha\underline{P}-\delta)\times \left([0,T]\setminus [t_2,t_3]\right),\\
 \partial_t (\log\eta_4)\geq M_4&\text{on }(-\infty,-\alpha\underline{P}-2\delta]\times \left([0,T]\setminus [t_2,t_3]\right).
 \end{cases}
 \end{equation}
Here $M_4$ 
is chosen such that
$$M_4>\|V_2\|_{L^{\infty}}+\tilde{\lambda}_{\rm min}+ \|\partial_t\log \overline{\phi}\|_{L^{\infty}},$$
where $\overline{\phi}$ is defined by \eqref{liu000025}.
Moreover,  we extend $\overline{\phi}$ to $(-\infty,1-\alpha\underline{P}+2\delta]\times[0,T]$ by setting
$\overline{\phi}(\cdot,t)\equiv\overline{\phi}(-\alpha\underline{P}-2\delta,t)$ on $ (-\infty,-\alpha\underline{P}-2\delta)\times[0,T]$,
so that by \eqref{liu000026} we have
\begin{equation}\label{liu119}
 \partial_y\overline{\phi}((-\alpha\underline{P}-2\delta)^+,\cdot)<0=\partial_y\overline{\phi}((-\alpha\underline{P}-2\delta)^-,\cdot).
\end{equation}

Let $\overline{\phi}$ and $\eta_4$ be given by \eqref{liu000025} and \eqref{eta4}, then we define
  \begin{equation}\label{defineliu001}
 \overline{\varphi}(y,t):=\eta_4(y,t)\cdot\overline{\phi}(y,t).
  \end{equation}
 By  \eqref{liu119}, as $\eta_4$ is smooth,  one can 
 infer that
$$\partial_y\log \overline{\varphi}\left((-\alpha\underline{P}-2\delta)^+,\cdot\right)<\partial_y\log\overline{\varphi}\left((-\alpha\underline{P}-2\delta)^-,\cdot\right)\quad\text{as }\,-\alpha\underline{P}-2\delta\in\mathbb{X}.$$

It remains to check that $\overline{\varphi}$ defined above satisfies \eqref{supersoltionprop4.1}.

\smallskip
\noindent {\bf (i)} For $y\in(-\alpha P(t),1-\alpha P(t))\cap[-\alpha\underline{P}-\delta,1-\alpha\overline{P}+\delta]$ and $t\in[0,T]$,
since $\eta_4\equiv 1$ in \eqref{eta4},
we have
$\overline{\varphi}(y,t)=\overline{\phi}(y,t)$.
By the definition of $\overline{\phi}$ in \eqref{liu000025}, direct calculations  yield that
 \begin{equation*}
 \begin{split}
 \smallskip
\partial_{t}\overline{\varphi}-D\partial_{yy}\overline{\varphi}+V_2\overline{\varphi}&=\left[\hat{V}_{2\epsilon}(y)-V_{2\epsilon}(y,t)+V_2(y,t)\right]\overline{\phi}-D\partial_{yy}\overline{\phi}.
 \end{split}
 \end{equation*}
By the definition of $V_{2\epsilon}$, we can argue as in Lemma \ref{sdlemdegen} to choose $D$ small such that  the first inequality in \eqref{supersoltionprop4.1} holds.
 Then the part of boundary conditions on $\{-\alpha P(t),1-\alpha P(t)\}\cap[-\alpha\underline{P}-\delta,1-\alpha\overline{P}+\delta]$ and $t\in[0,T]$  can be verified by \eqref{liu000026}.

\smallskip
\noindent {\bf (ii)} For $y\in(-\alpha P(t),1-\alpha P(t))\cap[-\alpha\underline{P}-2\delta,-\alpha\underline{P}-\delta)$ and $t\in[0,T]$,
 since $t\in[0,T]\setminus [t_2,t_3]$ in this case, we use \eqref{eta4} and \eqref{defineliu001} to deduce that
 \begin{equation*}
 \begin{split}
  \partial_{t}\overline{\varphi}-D\partial_{yy}\overline{\varphi}+V_2\overline{\varphi}
 &= \left[\hat{V}_{2\epsilon}(y)-V_{2\epsilon}(y,t)+V_2(y,t)\right]\overline{\varphi}+\left[\partial_t(\log\eta_4)-D\partial_{yy}\overline{\varphi}\right]\overline{\varphi}\\
 &\geq \left[\tilde{\lambda}_{\rm min}-2\epsilon+\partial_t(\log\eta_4)+O(D)\right]\overline{\varphi}.
 \end{split}
 \end{equation*}
 Since  $\partial_t(\log \eta_4)>0$ in this case, again we  choose $D$ small such  that \eqref{supersoltionprop4.1} holds.
And the boundary conditions in this case can be verified by $\partial_y\overline{\phi}\leq0$ and $\partial_y\eta_4\leq0$.

\smallskip
\noindent {\bf (iii)} For $y\in(-\alpha P(t),1-\alpha P(t))\cap(-\infty,-\alpha\underline{P}-2\delta)$ and $t\in[0,T]$, since  $\overline{\phi}$ is independent of $y$, by  \eqref{eta4} and \eqref{defineliu001} direct calculation yields that
\begin{equation*}
\begin{split}
 \partial_{t}\overline{\varphi}-D\partial_{yy}\overline{\varphi}+V_2\overline{\varphi}
 \geq\Big[ (\log \overline{\phi})'+M_4-D\partial_{yy}\eta_4/\eta_4+V_2\Big]\overline{\varphi}.
\end{split}
\end{equation*}
Thus the first inequality in \eqref{supersoltionprop4.1} is verified  by
the definition of $M_4$, and the boundary condition follows from $\partial_y\eta_4\leq 0$. Step 2 is now completed.

\medskip
{\bf Step 3.} Assume $\hat{b}=0$ and $\alpha>\frac{1}{\overline{P}-\underline{P}}$. We establish the second part of {\rm(ii)} in Theorem  \ref{sdthm4}.
 Let $\tilde{P}_\alpha$ denote the unique solution of \eqref{eq:tildeP}.  
We  apply the transformation $x=y+\tilde{P}_\alpha(t)$ to rewrite problem \eqref{SD_eq 4} as
\begin{equation*}
\begin{cases}
 \medskip
\partial_{t}\varphi-D\partial_{yy}\varphi- \alpha\tilde{b}(t)\partial_{y}\varphi+V_3\varphi
=\lambda(D)\varphi, \ \ &(y,t)\in \tilde{\Omega},\\
\medskip
 \partial_{y}\varphi(-\tilde{P}_\alpha(t),t)=\partial_{y}\varphi(1-\tilde{P}_\alpha(t),t)=0, \ \ & t\in [0,T], \\
\varphi(y,0)=\varphi(y,T), &y\in (-\tilde{P}_\alpha(0),1-\tilde{P}_\alpha(0)),
\end{cases}
 \end{equation*}
 where $\tilde{b}(t):=b(t)-F(\tilde{P}_\alpha(t),t)$, $V_3(y,t)=V(\tilde{P}_\alpha(t)+y,t)$, and
 $$\tilde{\Omega}=\left\{(y,t):y\in (-\tilde{P}_\alpha(t),1-\tilde{P}_\alpha(t)),\,t\in[0,T]\right\}.$$
 See Fig.\ref{figure3} for an example of this transformation.
 \begin{figure}[http!!]
  \centering
\includegraphics[height=1.4in]{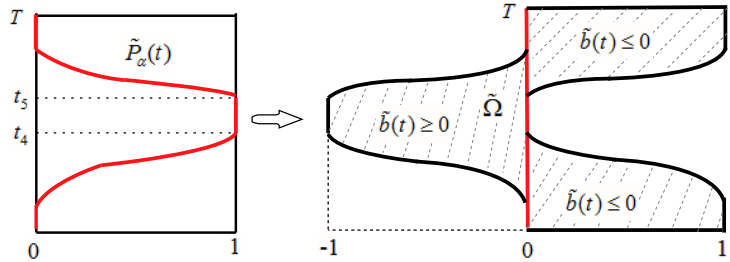}
  \caption{The diagram of $\tilde{\Omega}$ under transformation $x=y+\tilde{P}_\alpha(t)$. The red colored curve in the left side picture corresponds to $\tilde{P}_\alpha(t)$,  whereas the red colored line in the right side picture is the image of $\tilde{P}_\alpha(t)$ after the transformation.}\label{figure3}
  \end{figure}

  It remains to prove
   $$\lim\limits_{D\rightarrow0}\lambda(D)=\hat{V}_3(0).$$
The upper bound   $\limsup_{D\rightarrow 0}\lambda(D)\leq \hat{V}_3(0)$  can be established by using the arguments in Step 1 of Proposition  \ref{sdlem1}.  We next prove $\liminf_{D\rightarrow 0}\lambda(D)\geq \hat{V}_3(0)$.

We claim that if $\alpha>\frac{1}{\overline{P}-\underline{P}}$,   then
  $$ \mathrm{mes}\Big\{t\in[0,T]:\tilde{P}_\alpha(t)\in \{0,1\}\,\text{ and }\,b\neq 0\Big\}>0,$$
 i.e. there exist $0\leq t_4<t_5\leq T$ such that
$b\neq 0$, and $\tilde{P}_\alpha(t)\equiv0$ or $\tilde{P}_\alpha(t)\equiv1$ on $[t_4,t_5]$.
Suppose not, then $\tilde{P}_\alpha$ is also a periodic solution of $\dot{P}(t)=-\alpha b(t)$, so that $\tilde{P}_\alpha(t)=P(t)+c$ for  $c\in\mathbb{R}$, where $P(t)=-\int_{0}^{t}b(s)\mathrm{d}s$ as defined in part {\rm (ii)} of Theorem \ref{sdthm4}. Since $\tilde{P}_\alpha\in[0,1]$,
we have
 $$1\geq\max_{[0,T]}\tilde{P}_\alpha-\min_{[0,T]}\tilde{P}_\alpha=\alpha(\overline{P}-\underline{P}),$$
which contradicts $\alpha>\frac{1}{\overline{P}-\underline{P}}$.

In what follows, we assume $\tilde{P}_\alpha(t)\equiv1$ on $[t_4,t_5]$, and the proof is similar for the other case. To proceed further, we introduce positive functions $\overline{z}_5\in C^2(\mathbb{R})$ and $\eta_5\in C^1([0,T])$ as follows:
For any $\epsilon>0$, we  choose some small $\delta>0$ such that
\begin{equation}\label{defdeltathm1.1}
\begin{array}{ll}
|V_3(y,t)-V_3(0,t)|<\epsilon/2 \,\,\text{ on } \,[-2\delta,2\delta]\times[0,T].
\end{array}
\end{equation}
We first choose $\eta_5$ to be $T$-periodic and
\begin{equation}\label{eta3}
 (\log\eta_5)'>2\|V\|_{L^{\infty}}+\|(\log f_1)'\|_{L^{\infty}}
 \quad\text{on }\,\, [0,t_4+\delta]\cup[t_5-\delta,T].
\end{equation}
Then we choose  $\overline{z}_5$ such that
\begin{equation}\label{overlinez3}
\begin{cases}
\smallskip
\overline{z}_5'(y)<0 \,\,\text{ in }\,(-\infty,0),\,\,\,\,\overline{z}_5'(y)>0 \,\,\text{ in }\,(0,\infty),\\
(\log\overline{z}_5)'\leq-M_5\,\,\text{ in }\,(-\infty,-\delta),
\end{cases}
\end{equation}
where  $M_5$ is some large constant to be determined later.

We define 
\begin{equation}\label{liu008}
  \overline{\varphi}(y,t):=\overline{z}_5(y)\cdot\begin{cases}
f_1(t) &\text{for }|y|\leq\delta,\\
\zeta_5(y,t)&\text{for }-2\delta<y<-\delta,\\
\zeta_5(-y,t)&\text{for }\delta<y<2\delta,\\
\eta_5(t)&\text{for }|y|\geq 2\delta,
\end{cases}
\end{equation}
where $f_1$ is defined by \eqref{deff} with $x=1$.  Due to the choice of $\eta_5$ in \eqref{eta3}, $\zeta_5$ can be chosen such that $\overline{\varphi}\in C^{2,1}(\mathbb{R}\times[0,T])$ and
\begin{equation}\label{zeta4}
\begin{array}{l}
\partial_t(\log\zeta_5)\geq (\log f_1)' \quad\text{on }\,\, [0,t_4+\delta]\cup[t_5-\delta,T].
\end{array}
\end{equation}


We shall verify that $\overline{\varphi}$ defined by \eqref{liu008} satisfies
 \begin{equation}\label{supersoltionthm1.3}
L_D\overline{\varphi}:= \partial_{t}\overline{\varphi}-D\partial_{yy}\overline{\varphi}- \alpha\tilde{b}(t)\partial_{y}\overline{\varphi}+V_3\overline{\varphi}\geq (\hat{V}_3(0)-\epsilon)\overline{\varphi}\quad\text{for}\quad(y,t)\in \tilde{\Omega},
\end{equation}
provided that $D$ is small enough. The verification is divided into the following  cases:

\medskip
\noindent {\bf (i)} For $(y,t)\in([-\delta,\delta]\times[0,T])\cap\tilde{\Omega}$, we note that (see Fig.\ref{figure3})
$$\tilde{b}(t)\geq 0 \,\,\text{ in }\,\,([-\delta,0]\times[0,T])\cap\tilde{\Omega}\quad\text{and}\quad\tilde{b}(t)\leq0 \,\,\text{ in }\,\,([0,\delta]\times[0,T])\cap\tilde{\Omega}.$$
One can  check \eqref{supersoltionthm1.3} by the same arguments as in Step 2 of Proposition \ref{sdlem1}.

\medskip

\noindent {\bf (ii)} For $(y,t)\in((-\infty,-\delta]\times[t_4+\delta,t_5-\delta])\cap\tilde{\Omega}=(-1,-\delta)\times[t_4+\delta,t_5-\delta]$ (since $\tilde{P}_\alpha(t)\equiv1$ on $[t_4,t_5]$), there exists some $\epsilon_0>0$ such that $\tilde{b}(t)>\epsilon_0$. By the choice of $\overline{z}_5$ in \eqref{overlinez3} and construction \eqref{liu008},  direct calculation gives
\begin{equation*}
\begin{split}
L_D\overline{\varphi}\geq  &\Big[-\left|(\log\eta_5)'+\partial_t(\log \zeta_5)'\right|
-D\partial_{yy}\overline{\varphi}+\alpha\epsilon_0M_5-\alpha \tilde{b}\left|\zeta_5\right|+V_3\Big]\overline{\varphi}.
\end{split}
\end{equation*}
By choosing $M_5$  large and  $D$ small, we can verify that \eqref{supersoltionthm1.3} holds.

\medskip

\noindent {\bf (iii)} For $(y,t)\in\left([-2\delta,-\delta]\times([0,t_4+\delta]\cup[t_5-\delta,T])\right)\cap\tilde{\Omega}$, by construction, $\overline{\varphi}(y,t)=\overline{z}_5(y)\zeta_5(y,t)$. Observe that $\tilde{b}\geq 0$ in this case. Using \eqref{overlinez3}, we choose  $M_5$ large  such that
$$-\tilde{b}(t)\partial_y\overline{\varphi}\geq\tilde{b}(t)\left[M_5-\partial_y(\log\zeta_5)\right]\overline{\varphi}\geq 0.$$
 Hence, by  \eqref{defdeltathm1.1} and \eqref{zeta4}, for small $D$ we arrive at
\begin{equation*}
\begin{split}
L_D\overline{\varphi}&\geq \Big[\partial_t(\log\zeta_5)
+V_3-\epsilon/2\Big]\overline{\varphi}\\
&\geq \Big[(\log f_1)'+V_3-\epsilon/2\Big]\overline{\varphi}\\
&= \Big[\hat{V}_3(0)-V_3(0,t)+V_3(y,t)-\epsilon/2\Big]\overline{\varphi}\\
&\geq \Big[\hat{V}_3(0)-\epsilon\Big]\overline{\varphi}.
\end{split}
\end{equation*}

\smallskip
\noindent {\bf (iv)} For $(y,t)\in\left((-\infty,-2\delta)\times([0,t_4+\delta]\cup[t_5-\delta,T])\right)\cap \,\tilde{\Omega}$, by \eqref{liu008} we have $\overline{\varphi}(y,t)=\overline{z}_5(y)\eta_5(t)$. Also since $\tilde{b}\geq 0$, the choice of $\overline{z}_5$ in \eqref{overlinez3} implies $-\tilde{b}(t)\partial_y\overline{\varphi}\geq 0$.  Choosing $D$  smaller if necessary, we use \eqref{eta3} to deduce  that
\begin{equation*}
\begin{split}
L_D\overline{\varphi}&\geq\Big[(\log\eta_5)'-D\overline{z}''_5/\overline{z}_5-V_3\Big]\overline{\varphi}\geq \hat{V}_3(0)\overline{\varphi}.
\end{split}
\end{equation*}

\noindent {\bf (v)} For $(y,t)\in((\delta,\infty)\times[0,T])\cap\tilde{\Omega}$, the verification of \eqref{supersoltionthm1.3} is rather similar to that in cases {\bf (ii)}-{\bf (iv)}, and thus is omitted.

 Finally,  we verify the boundary conditions
 \begin{equation}\label{liu000031}
 \partial_{y}\overline{\varphi}(-\tilde{P}_\alpha(t),t)\leq0\quad\text{and}\quad \partial_{y}\overline{\varphi}(1-\tilde{P}_\alpha(t),t)\geq0 \quad\text{for}\quad  t\in [0,T].
 \end{equation}
For the set $\{t\in[0,T]: -\alpha\tilde{P}_\alpha(t)\in[-2\delta,-\delta]\,\text{ or }\, 1-\alpha\tilde{P}_\alpha(t)\in[\delta, \,2\delta]\}$,
 we can choose   $M_5$ large such that $M_5>\|\partial_y(\log\zeta_5)\|_{L^\infty}$  to verify \eqref{liu000031} as in case {\bf (iii)}. The verification of \eqref{liu000031} for the remaining cases is straightforward.

By \eqref{supersoltionthm1.3} and \eqref{liu000031}, we apply Proposition \ref{appendixprop} and Remark  \ref{appendixrem} to conclude $\liminf_{D\rightarrow0}\lambda(D)\geq\hat{V}_3(0)$.
The proof of Theorem  \ref{sdthm4} is thereby completed.
%
%
\end{proof}

\begin{appendices}
\section{Generalized super/sub-solution for a periodic parabolic operator}
In this section, we introduce a generalized definition of super/sub-solution for a time-periodic parabolic operator
and then present a comparison result. 
This result is a mortification of Proposition A.1 in \cite{LLPZ20192}, and  it plays a vital role in 
this paper.

Let $\mathcal{L}$ denote the following linear parabolic operator over $(0,1)\times[0,T]$:
$$\mathcal{L}=\partial_{t}\varphi-a_1(x,t)\partial_{xx}- a_2(x,t)\partial_x+a_0(x,t).$$
In the sequel, we always assume $a_1(x,t)>0$ so that $\mathcal{L}$ is uniformly elliptic for each $t\in[0,T]$, and
assume  $a_0,a_1,a_2\in C([0,1]\times[0,T])$ are $T$-periodic in $t$.

Consider the linear parabolic problem
\begin{equation}\label{appendix}
\tag{A.1}
 \begin{cases}
 \smallskip
\mathcal{L}\varphi=0\ \ &{\text{in}}\,\,(0,1)\times[0,T],\\
\smallskip
 c_1\partial_x\varphi(0,t)-(1-c_1)\varphi(0,t)=0\ \ & {\text{on}}\,\,[0,T], \\
 \smallskip
c_2\partial_x\varphi(0,t)+(1-c_2)\varphi(1,t)=0 \ \ & {\text{on}}\,\,[0,T], \\
 \varphi(x,0)=\varphi(x,T) &{\text{on}}\,\,(0,1),
 \end{cases}
 \end{equation}
 where $c_1,c_2\in[0,1]$.
We now define the  super/sub-solution corresponding to \eqref{appendix} as follows.

\begin{definition}\label{appendixldef}
The function $\overline{\varphi}$ in $[0,1]\times[0,T]$ is  called a  super-solution of \eqref{appendix} if there  exists a  set $\mathbb{X}$  consisting of at most finitely many points: 
$$\mathbb{X}=\emptyset \,\,\text{ or }\,\,\mathbb{X}=\left\{\kappa_i\in (0,1):\ \, i=1,\ldots,N\right\}$$
for some  integer $N\geq1$, such that 
\begin{itemize}
  \item [(i)]  $\overline{\varphi}\in C\left((0,1)\times[0,T]\right)\cap C^2\left(((0,1)\setminus\mathbb{X})\times[0,T]\right);$
  \item [(ii)] $\partial_x\overline{\varphi}(x^+,t)<\partial_x\overline{\varphi}(x^-,t)$ \,\,\, for every\, $x\in\mathbb{X}$ \,and \,$t\in[0,T]$;
  \item [(iii)]  $\overline{\varphi}$ satisfies
  \begin{equation*}
 \begin{cases}
 \smallskip
\mathcal{L}\overline{\varphi}\geq 0\ \ &{\mathrm{ in }}\,\,((0,1)\setminus\mathbb{X})\times(0,T),\\
\smallskip
  c_1\partial_x\varphi(0,t)-(1-c_1)\varphi(0,t)\leq 0 &{\mathrm{ on }}\,\,[0,T], \\
  \smallskip
 c_2\partial_x\varphi(1,t)+(1-c_2)\varphi(1,t)\geq 0  &{\mathrm{ on }}\,\,[0,T], \\
 \overline\varphi(x,0)\geq\overline\varphi(x,T) &\mathrm{ on } \,\,(0,1).
  \end{cases}
 \end{equation*}
\end{itemize}
A super-solution $\overline{\varphi}$ is called to be strict  if it is not a solution of \eqref{appendix}. Moreover,
a function $\underline{\varphi}$ is called a (strict) sub-solution of \eqref{appendix} if $-\underline{\varphi}$ is a  (strict) super-solution.
\end{definition}

Let $\lambda(\mathcal{L})$ denote the principal eigenvalue of  the problem
\begin{equation}\label{A.3}
 \tag{A.2}
 \begin{cases}
 \smallskip
\mathcal{L}\varphi=\lambda(\mathcal{L})\varphi\ \ &{\text{in}}\,\,(0,1)\times[0,T],\\
\smallskip
 c_1\partial_x\varphi(0,t)-(1-c_1)\varphi(0,t)=0\ \ & {\text{on}}\,\,[0,T], \\
 \smallskip
c_2\partial_x\varphi(0,t)+(1-c_2)\varphi(1,t)=0 \ \ & {\text{on}}\,\,[0,T], \\
 \varphi(x,0)=\varphi(x,T) &\text{on}\,\,(0,1).
 \end{cases}
 \end{equation}

The following result was proved in \cite[Proposition A.1]{LLPZ20192} for the case $c_1=c_2=1$, and it can be extended to the general case $c_1,c_2\in[0,1]$.

\begin{prop}\label{appendixprop}
Let $\lambda(\mathcal{L})$ denote the principal eigenvalue of \eqref{A.3}.
If \eqref{appendix} admits some strict positive super-solution defined in Definition {\rm \ref{appendixldef}}, then
$\lambda(\mathcal{L})\geq0$. Moreover, if \eqref{appendix} admits some strict  nonnegative sub-solution  defined in Definition {\rm \ref{appendixldef}}, then $\lambda(\mathcal{L})\leq0$.
\end{prop}

\begin{remark}\label{appendixrem}
{\rm Instead of $[0,1]\times[0,T]$, Proposition {\rm \ref{appendixprop} } also holds for the general domain given by
$\left\{(x,t): \beta_1(t)<x<\beta_2(t),\,t\in[0,T]\right\}$,
where $\beta_1,\beta_2\in C([0,T])$ satisfy $\beta_1<\beta_2$. This fact is applied in Section {\rm \ref{S4}} to prove Theorem {\rm \ref{sdthm4}}.}
\end{remark}
\end{appendices}
\bigskip
\noindent{\bf Acknowledgments.} We sincerely thank the referees for their suggestions which help improve the manuscript.
SL was partially supported by the Outstanding Innovative Talents Cultivation Funded Programs 2018 of Renmin Univertity of China and the NSFC grant No. 11571364;
YL  was partially supported by the NSF grant DMS-1853561;
RP was partially supported by the NSFC grant No. 11671175;
MZ was partially supported by Nankai ZhiDe Foundation.

\end {document}